\documentclass[11pt, a4paper, reqno]{amsart}

\usepackage{needspace}

\usepackage[T1]{fontenc}
\usepackage{palatino}
\usepackage{eucal}
\usepackage{hyperref}
\usepackage{amssymb,upref,enumerate}
\usepackage{tikz}
\usepackage{amsmath,amssymb,yfonts,amscd,amsthm,amsxtra,mathrsfs,mathabx}
\usepackage{latexsym}
\usepackage{esint}
\usepackage{mathtools}

\usepackage{resmes}

\usepackage{xcolor}
\usepackage{amsthm}

\usepackage{framed}

\colorlet{shadecolor}{gray!15}

\def\XXint#1#2#3{{\setbox0=\hbox{$#1{#2#3}{\int}$} 
\vcenter{\hbox{$#2#3$}}\kern-.5\wd0}}

\newcommand{\ra}{\rightarrow}

\newcommand{\bey}{\begin{eqnarray*}}
\newcommand{\eey}{\end{eqnarray*}}
\newcommand{\ba}{\begin{align}}
\newcommand{\ea}{\end{align}}
\newcommand{\bea}{\begin{align*}}
\newcommand{\ena}{\end{align*}}
\newcommand{\be}{\begin{equation}}
\newcommand{\ee}{\end{equation}}
\newcommand{\R}{\mathbb R}
\newcommand{\Z}{\mathbb Z}

\newcommand{\Dy}{\mathscr D}
\newcommand{\N}{\mathbb N}

\newcommand{\D}{\mathcal D}

\newcommand{\A}{\textup A}

\newcommand{\Haus}{\mathcal H }

\newcommand{\ep}{\epsilon}

\newcommand{\bc}{\begin{center}}
\newcommand{\ec}{\end{center}}

\newcommand{\vp}{\varphi}

\newcommand{\al}{\alpha}

\newcommand{\dx}{\,{\mathrm dx}}
\newcommand{\dmu}{\,{\mathrm d\mu}}

\newcommand{\dy}{\,{\mathrm dy}}
\newcommand{\dt}{\,{\mathrm dt}}

\newcommand{\Lup}{\textup{L}}
\newcommand{\Leb}{\mathcal{L}}
\newcommand{\BV}{\textup{BV}}
\newcommand{\Ws}{\textup{W}}

\newcommand{\Ldash}{\textup{\L}}
\newcommand{\Con}{\textup{C}}
\newcommand{\Har}{\textup{H}}

\DeclareMathOperator{\Lip}{\textup{Lip}}
\DeclareMathOperator{\Div}{\textup{div}}
\DeclareMathOperator{\Per}{\textup{Per}}

\DeclareMathOperator{\supp}{\textup{supp}}

\newtheorem{theorem}{Theorem}[section]
\newtheorem{lemma}[theorem]{Lemma}
\newtheorem{corollary}[theorem]{Corollary}

\theoremstyle{definition}

\newtheorem{conj}[theorem]{Conjecture}

\theoremstyle{remark}
\newtheorem{remark}[theorem]{Remark}

\begin{document}


\subjclass[2020]{Primary 46E35; Secondary 26D10, 42B25, 49Q20}

\keywords{Weighted Sobolev inequalities, Meyers--Ziemer inequality, maximal function, isoperimetric inequalities, endpoint estimates}

\title[Weighted Sobolev Inequalities]{Weighted Sobolev Inequalities via the Meyers--Ziemer Framework:
Measures, Isoperimetric Inequalities, and Endpoint Estimates}
\date{\today}
\author{Simon Bortz}

\address{Simon Bortz \\
 Department of Mathematics \\
 University of Alabama \\
 Tuscaloosa, AL 35487, USA}
\thanks{S. Bortz was supported by the Simons Foundation’s Travel Support for Mathematicians program and an AWI-
CONSERVE Fellowship.}
\email{sbortz@ua.edu}

\author{Kabe Moen}

\address{Kabe Moen \\
 Department of Mathematics \\
 University of Alabama \\
 Tuscaloosa, AL 35487, USA}

\email{kabe.moen@ua.edu}

\author{Andrea Olivo}

\address{Andrea Olivo\\
 Department of Mathematics \\
 Facultad de Ciencias Exactas y Naturales, Universidad de Buenos Aires and IMAS-CONICET\\
 Pabell\'on I (C1428EGA), Ciudad de Buenos Aires,
Argentina}
 \thanks{A. Olivo is supported by grants PID2023-146646NB-I00 funded by MICIU/AEI/10.13039/501100011033 and by ESF+, the Basque Government through the BERC 2022–2025 program, and through BCAM Severo Ochoa accreditation CEX2021-001142-S / MICIN / AEI / 10.13039/501100011033.}
\email{aolivo@dm.uba.ar}

\author{Carlos P\'erez }

\address{Carlos P\'erez\\
 Department of Mathematics \\
 University of the Basque Country \\
 Ikerbasque and BCAM \\
 Bilbao, Spain}
\thanks{C. P\'erez is supported by the Spanish government through the grant PID2023-146646NB-I00 and by Severo Ochoa accreditation CEX2021-001142-S both at BCAM and also by the Basque Government through grant IT1615-22 at the University of the Basque Country and by the BERC programme 2022-2025 at BCAM}
\email{cperez@bcamath.org}

\author{Ezequiel Rela} 
\address{Ezequiel Rela\\ 
Department of Mathematics\\
University of Buenos Aires and IMAS- CONICET-UBA
}
\email{erela@dm.uba.ar}

\maketitle

 \begin{abstract} We establish a new global endpoint Sobolev inequality for measures that extends the classical theorem of Meyers–Ziemer by placing a maximal function on the right-hand side. This result has several significant consequences. It extends naturally to functions of weighted bounded variation and yields corresponding capacity and isoperimetric inequalities. The inequality is also closely connected to endpoint estimates for fractional operators, including bounds for fractional maximal functions and Hardy space endpoint estimates for the Riesz potential. Our main inequality yields a family of endpoint inequalities, characterized in terms of subrepresentation formulas, Lorentz space improvements, and isoperimetric inequalities for measures and bounded open sets. When one moves away from the endpoint to $p>1$, the analogous inequalities no longer hold in general; however, we identify a sharp bumped maximal function for which the corresponding non-endpoint inequality is valid. Finally, we show that this framework yields new $(p,p)$ two-weight Sobolev inequalities.
 \end{abstract}

\section{Introduction}

\subsection{Classical Sobolev and Isoperimetric Inequalities}
Consider the classical Gagliardo-Nirenberg-Sobolev (GNS) inequality:
\begin{equation}\label{GNS}
\|u\|_{\Lup^{p^*}(\R^n)} \le s_{n,p} \|\nabla u\|_{\Lup^p(\R^n)},
\end{equation}
which holds for $1 \le p < n$, $p^* = \frac{np}{n-p}$, and $u \in \Con_c^\infty(\R^n)$.  
If one takes $\dot\Ws^{1,p}(\R^n)$ to be the completion of $\Con_c^\infty(\R^n)$ with respect to the norm $\|\nabla u\|_{\Lup^p(\R^n)}$ (which is not a norm in general, but it is a norm on $\Con_c^\infty(\R^n)$), then one obtains the Sobolev embedding
$\Ws^{1,p}(\R^n) \hookrightarrow \Lup^{p^*}(\R^n).$

The endpoint GNS inequality,
\begin{equation}\label{GN}
\|u\|_{\Lup^{\frac{n}{n-1}}(\R^n)} \le s_{n,1} \|\nabla u\|_{\Lup^1(\R^n)},
\end{equation}
is particularly important. First, it implies the general case $1<p<n$ by scaling the function. More importantly, it is equivalent to the isoperimetric inequality:
\[
\Per(\Omega) \ge c_n \Leb(\Omega)^{\frac{n-1}{n}},
\]
where $\Per(\Omega)$ denotes the perimeter of the set $\Omega$ and $\mathcal L$ is Lebesgue measure on $\R^n$.  
For sufficiently regular sets, the perimeter coincides with the Hausdorff measure of the boundary~$\partial \Omega$.
The sharp constant in the Sobolev inequality was found by Talenti \cite{Tal} and gives the best constant in the isoperimetric equality via 
\[
c_n = s_{n,1}^{-1} = n v_n^{1/n},
\]
where 
\[
v_n = \Leb(B_1(0)) = \frac{\pi^{n/2}}{\Gamma\big(\frac{n}{2}+1\big)},
\]
is the volume of the unit ball in $\R^n$.

Estimate \eqref{GN} was extended in \cite{PR2} to the setting of locally finite Borel measures $\mu$ as follows: there exists a dimensional constant $c_n$ such that, for any locally finite Borel measure $\mu$ and any Lipschitz function $u$ with compact support,
\begin{equation}\label{GN-mu}
\left(\int_{\mathbb{R}^n} |u(x)|^{n'}\dmu(x) \right)^{\frac1{n'}}
\leq c_n \int_{\mathbb{R}^n} |\nabla u(x)| (\textup{M}\mu(x))^\frac{1}{n'}\dx,
\end{equation}
where $\textup{M}$ denotes the usual Hardy--Littlewood maximal operator and $n'=\frac{n}{n-1}=1^*$. As a consequence of this estimate, an isoperimetric inequality involving measures was also established in \cite{PR2}: if $\Omega \subseteq \mathbb{R}^n$ is a bounded domain with smooth boundary $\partial \Omega$ and $\mu$ is a measure, then
\begin{equation}\label{isoperintro}
\mu(\Omega)^{\frac{1}{n'}} \leq c_n \int_{\partial \Omega} (\textup{M}\mu)^\frac{1}{n'} \, \mathrm{d}\mathcal H^{n-1},
\end{equation}
where $\Haus^{n-1}$ denotes the $(n-1)$-dimensional Hausdorff measure (see Section \ref{prelim}). The main objectives of this paper are to show that \eqref{GN-mu} is in fact part of a broader family of endpoint measure-theoretic Sobolev inequalities, and to extend the isoperimetric inequality \eqref{isoperintro} to more general sets by substantially relaxing the smoothness assumption on the boundary.

We emphasize that the constant $c_n$ above is not necessarily the same as in \eqref{GN-mu}. It would be interesting to determine whether the sharp constant in this setting agrees with the classical isoperimetric constant, namely in the case $\mu=\Leb$.

\subsection{Weighted and Two-Weight Sobolev Inequalities}

We are interested in GNS inequalities under a change of measure. As motivation, consider a quasiconformal change of variables in inequality \eqref{GNS}. Let $f: \R^n \to \R^n$ be a bijective map in $\Ws_{\textup{loc}}^{1,1}(\R^n,\R^n)$ satisfying
\[
|Df(x)|^n \le K J_f(x),
\]
where $Df(x)$ is the Jacobian matrix, $|Df(x)|$ its operator norm, and $J_f(x) = \det Df(x)$ the Jacobian determinant.  
After the change of variables inequality \eqref{GNS} becomes
\[
\left( \int_{\R^n} |u(x)|^{p^*} J_f(x) \dx \right)^{\frac1{p^*}} 
\le s_{n,p} K^{\frac1n} \left( \int_{\R^n} |\nabla u(x)|^p J_f(x)^{1 - \frac{p}{n}} \dx \right)^{\frac1p}, \qquad u \in \Con_c^\infty(\R^n).
\]
This leads naturally to the study of general weighted Sobolev inequalities of the form
\begin{equation}\label{weightedSob}
\|u\|_{\Lup^{p^*}(w)} \le C \|\nabla u\|_{\Lup^p(w^{1-\frac{p}{n}})}, \qquad 1\leq p<n.
\end{equation}
In general, the precise conditions on a weight $w$ under which \eqref{weightedSob} holds are difficult to characterize.  

One approach, introduced by David and Semmes \cite{DS} (see also Semmes \cite{Sem}), is based on the notion of \emph{strong \(A_\infty\) weights}, also referred to as \emph{metric doubling measures}. We will return to the work of David and Semmes in Section \ref{oneweight}.

An alternative approach to that of David and Semmes is the following. Consider a weight $w$ such that, for some constant $c_w$, the following estimate holds for any Lipschitz function $u$ with compact support:
\begin{equation}\label{weak-key estimate}
\|u\|_{\Lup^{n',\infty}(w)} \le c_{w} \|\nabla u\|_{\Lup^1(w^\frac1{n'})}.
\end{equation}
Then, by the truncation method \cite{Maz} (see Theorem \ref{weakstrong} in Section \ref{twoweight}), there exists a universal constant $\lambda$ such that we can lift this to the strong estimate \eqref{weightedSob}, namely:
\[
\|u\|_{\Lup^{n'}(w)} \le c_{w}\lambda \|\nabla u\|_{\Lup^1(w^\frac1{n'})}.
\]
Now, using a scaling argument (Theorem \ref{wp1imply} in Section \ref{oneweight}), we obtain
\begin{equation}\label{Sob-cond-w}
\|u\|_{\Lup^{p^*}(w)} \le c_{w,p} \|\nabla u\|_{\Lup^p(w^{1-\frac{p}{n}})}, \qquad 1 \le p < n.
\end{equation}
Thus, the fundamental issue is to identify the minimal condition required on $w$ to derive the weak type estimate \eqref{weak-key estimate}. This is formulated as Conjecture \ref{Ainftyconj} in Section \ref{oneweight}.

Now, in view of \eqref{GN-mu}, it is natural to ask whether an analogue of \eqref{GN-mu} remains valid for the exponent pair $(p,p^*)$ with $p>1$. In other words, can one establish an $\Lup^p$ version of \eqref{GN-mu}, namely
\begin{equation}\label{Sob-mu}
\left(\int_{\mathbb{R}^n} |u(x)|^{p^*}w(x)\dx) \right)^{\frac1{p^*}}
\leq c_{n,p} \left(\int_{\mathbb{R}^n} |\nabla u(x)|^p \textup{M}w(x)^{1-\frac{p}{n}}\dx\right)^{\frac1{p}}\text{?}
\end{equation}
In fact, this estimate is false. One must instead replace the operator $\textup{M}$ by larger iterated operators $\textup{M}^k$, where the integer $k$ depends on $p$ (see Theorem \ref{bumpreplacep>1}). We refer the reader to Section \ref{twoweight} for further discussion and related results.

On the other hand, it is possible to deduce the following estimate by using \eqref{GN-mu} and a scaling variant of the argument described above (see the proof of Theorem \ref{wp1imply}). The resulting inequality applies to weights rather than general measures. Specifically, for any weight $w$ on $\mathbb{R}^n$ and $1 \le p < n$, we have
\begin{equation}\label{eq:P(p,p*)-local-avg}
\left(\int_{\R^n} |u(x)|^{p^*}\,w(x)\dx \right)^{\frac1{p^*}}
\le c_{n,p} \left( \int_{\R^n} |\nabla u(x)|^p 
\frac{\textup{M}w(x)^{\frac{p}{n'}}}{w(x)^{p-1}} \dx \right)^{\frac{1}{p}}.
\end{equation}
Observe that the weight on the right-hand side is pointwise larger than $w^{1-\frac{p}{n}}$ appearing in \eqref{Sob-cond-w}. Inequality \eqref{eq:P(p,p*)-local-avg} follows as a corollary of \cite[Theorem 1.21]{PR1}, where a localized version was established using a localized form of \eqref{GN-mu}.

To prove inequality \eqref{weak-key estimate}, we can invoke the classical theory of Muckenhoupt weights via the Riesz potential operator
\[
\textup{I}_\al f(x)
= \frac{1}{\gamma(\al)} \int_{\R^n} \frac{f(y)}{|x-y|^{n-\al}}\dy,
\qquad 0 < \al < n,
\]
where
\begin{equation*}\label{fracconst}
\gamma(\al)
= \frac{2^\al \pi^{n/2} \Gamma(\al/2)}{\Gamma((n-\al)/2)}.
\end{equation*}

The Riesz potential $\textup{I}_1$ plays a central role in Sobolev inequalities due to the subrepresentation formula
\begin{equation}\label{subrepresentation}
|u(x)| \le \frac{\gamma(1)}{\omega_{n-1}}\,\textup{I}_1(|\nabla u|)(x),
\qquad u \in \Lip_c(\R^n),
\end{equation}
where $\omega_{n-1} = n v_n$ denotes the surface measure of the unit sphere.  
Although \eqref{subrepresentation} is usually stated for $u \in \Con_c^\infty(\R^n)$, it remains valid for $u \in \Lip_c(\R^n)$ and for all $x \in \R^n$.  Muckenhoupt and Wheeden \cite{MW} characterized the weights for which the Riesz potential is bounded by showing that the inequality
\[
\|\textup{I}_\alpha f\|_{\Lup^q(w)} \le C \|f\|_{\Lup^p(w^{\frac{p}{q}})},
\qquad
{1< p}<\tfrac{n}{\alpha},
\quad
\tfrac1q=\tfrac1p-\tfrac{\alpha}{n},
\]
holds if and only if $w$ belongs to the Muckenhoupt class $\A_{1+\frac{q}{p'}}$.  At the endpoint, $p=1$,  the weak type estimate
$$\|\textup I_\al f\|_{\Lup^{\frac{n}{n-\al},\infty}(w)}\leq C\|f\|_{\Lup^1(w^{\frac{n}{n-\al}})}$$ holds
if and only if $w\in \A_1$. Combining \eqref{subrepresentation} with this result in the case $\alpha=1$ and $q=p^*$ shows that the weighted Sobolev inequality \eqref{weightedSob} holds whenever
$w\in \A_{\frac{p^*}{n'}}$. Indeed, a straightforward computation gives
$1+p^*/p' = p^*/n' = p(n-1)/(n-p)$.  As a concrete example, for power weights one obtains the inequality
\[
\|u\|_{\Lup^{p^*}(|x|^\lambda)} \le C \|\nabla u\|_{\Lup^p(|x|^{\lambda(1-p/n)})},
\]
valid for $-n<\lambda< n p^*/p'$. In Section~\ref{oneweight} we study the one-weight Sobolev inequality \eqref{weightedSob} in greater detail.

Consider more general two-weight Sobolev inequalities of the form
\begin{equation}\label{twowSob}
\|u\|_{\Lup^q(w)} \le C\,\|\nabla u\|_{\Lup^p(v)}.
\end{equation}
Such inequalities also play a central role in quantitative uncertainty principles {(see 
\cite{CW, Feff, CWW, LL, Per95}).}  In the case $1 < p \le q < \infty$, inequality \eqref{twowSob} is a consequence of Sawyer’s weak boundedness theorem for $\textup{I}_\al$ \cite{Saw84}. 
Specifically, Sawyer proved that the weak type inequality 
\[
\|\textup{I}_\al f\|_{\Lup^{q,\infty}(w)} \le C\,\|f\|_{\Lup^p(v)}
\]
holds if and only if the testing condition
\begin{equation}\label{testing}
\left(\int_Q \textup{I}_\al(\mathbf 1_Q w)^{p'}\, v^{1-p'}\,\mathrm dx\right)^{\frac{1}{p'}} \le C\,w(Q)^{\frac1{q'}}
\end{equation}
holds for all cubes $Q$. Thus, if $(w,v)$ satisfy \eqref{testing} with $\al=1$, then the weak Sobolev inequality holds via \eqref{subrepresentation},
\[
\|u\|_{\Lup^{q,\infty}(w)} \le C\|\textup{I}_1(|\nabla u|)\|_{\Lup^{q,\infty}(w)}\leq C\,\|\nabla u\|_{\Lup^p(v)}.
\]
This can be upgraded to the strong form given in \eqref{twowSob} by the coarea formula or Maz’ya’s truncation method; again we refer to Theorem \ref{weakstrong}.

While the condition \eqref{testing} is necessary and sufficient for the weak boundedness of the Riesz potential—and thus sufficient for the two weight Sobolev inequality—it is not clear that it actually \emph{characterizes} the Sobolev inequality. Moreover, it is somewhat unsatisfactory that the sufficient condition involves the operator $\textup{I}_1$, whereas the inequality itself concerns a function and its gradient.

Thus, a different approach, one initiated by the fourth author, is to seek sharp sufficient conditions, known as bump conditions, involving only the weights $(w,v)$ and not the operator $\textup{I}_1$. One can verify that a necessary condition for the mapping $\textup{I}_\al:\Lup^p(v)\to L^{q,\infty}(w)$ is the two-weight Muckenhoupt condition
\begin{equation}\label{Apq}
\sup_Q \ell(Q)^\al\,\Leb(Q)^{\frac1q-\frac1p}
\left(\fint_Q w\,\mathrm dx\right)^{\frac1q}
\left(\fint_Q v^{1-p'}\,\mathrm dx\right)^{\frac1{p'}}
<\infty.
\end{equation}
Here we are using the standard notation of an integral average
$$\fint_Q f\dx=\frac{1}{\mathcal L(Q)}\int_Q f(x)\dx.$$
A natural direction, then, is to enlarge the averages in this condition in a way that remains both sharp and tractable—this leads to the so-called bump conditions. The fourth author \cite{Per94} showed that assuming sufficiently bumped averages on both weights yields the strong-type inequality $\textup{I}_\al:\Lup^p(v)\to \Lup^q(w)$, and therefore the Sobolev inequality. In the strict upper-triangular case $1<p<q<\infty$, this was refined by the second author and Cruz-Uribe \cite{CM}, who proved that if $(w,v)$ satisfy the logarithmic bump (see Section \ref{prelim} for definitions of the logarithmic averages)
\begin{equation}\label{logbump}
\sup_Q \ell(Q)^\al\,\Leb(Q)^{\frac1q-\frac1p}
\|w\|_{\Ldash(\log \Ldash)^{\frac{q}{p'}+\ep}(Q)}
\left(\fint_Q v^{1-p'}\,\mathrm dx\right)^{\frac1{p'}}
<\infty,
\end{equation}
for some $\ep>0$, then $\textup{I}_\al:\Lup^p(v)\to \Lup^{q,\infty}(w)$. Taking $\al=1$, we see that \eqref{twowSob} follows from \eqref{subrepresentation} and Theorem \ref{weakstrong}. We also show that the condition \eqref{logbump} is sharp for the Sobolev inequality \eqref{twowSob} in the sense that one cannot take $\ep=0$ in \eqref{logbump} and still obtain a sufficient condition.

In the diagonal case $p=q$, the optimal bump conditions remain unknown. In fact, the best currently available sufficient bump condition for the $(p,p)$ two weight Sobolev inequality
\begin{equation*}\label{pptwowSob}
\|u\|_{\Lup^p(w)} \le C\,\|\nabla u\|_{\Lup^p(v)}
\end{equation*}
is given by
\begin{equation}\label{nonoptlogbump}
\sup_Q \ell(Q)\,
\|w\|_{\Ldash(\log \Ldash)^{2p-1+\ep}(Q)}
\left(\fint_Q v^{1-p'}\,\mathrm dx\right)^{\frac1{p'}}
<\infty.
\end{equation}
The astute reader will observe that taking $p=q$ in \eqref{logbump} suggests a logarithmic power of $\frac{p}{p'}+\ep = p-1+\ep$, which is strictly smaller than $2p-1+\ep$. We will revisit two-weight Sobolev inequalities in Section \ref{twoweight} and show that, in the diagonal case $p=q$, the bump conditions can indeed be improved to the optimal ones (see Theorem \ref{optimalppbump}).

\subsection{Endpoint Inequalities and Maximal Operators}

It is well known that endpoint inequalities often imply their non-endpoint counterparts. As noted above, any estimate for the Riesz potential yields a corresponding Sobolev inequality. 
{For this reason, establishing sharp endpoint bounds for Riesz potentials remains} an active area of research. To illustrate the subtleties involved, we begin by recalling the analogous endpoint bounds for the fractional maximal function, defined for a locally finite Borel measure
\[
(\textup{M}_\alpha \mu)(x)
= \sup_{r>0} r^\alpha \fint_{B_r(x)}  \dmu,
\qquad 0 \le \alpha < n.
\]
When $\dmu=|f|\dx$ for $f\in\Lup_{\text{loc}}^1(\R^n)$, we write $\textup{M}_\al f$.  When $\alpha=0$, $\textup{M}_0=\textup{M}$ is the Hardy--Littlewood maximal operator. For $0<\alpha<n$, $\textup{M}_\alpha$ is the natural maximal operator associated with the Riesz potential, and in fact
\[
\textup{M}_\alpha f(x)\le c_{n,\alpha}\,\textup{I}_\alpha f(x),
\qquad f\ge 0.
\]

A classical result of {Fefferman and Stein}~\cite{St} for the Hardy--Littlewood maximal operator asserts that the weak-type inequality
\begin{equation}\label{WeakMM}
\|\textup{M}f\|_{\Lup^{1,\infty}(w)}
\le { C_n}\,\|f\|_{\Lup^1(\textup{M}w)}
\end{equation}
holds for any weight $w$ {or measure}. This estimate has applications in harmonic analysis and PDE, as it implies vector-valued bounds for $\textup{M}$ that play a central role in Littlewood--Paley theory. Sawyer~\cite{Saw81} later showed that inequality~\eqref{WeakMM} extends to the fractional setting, namely,
\begin{equation} \label{Malweak}
\|\textup{M}_\alpha f\|_{\Lup^{1,\infty}(w)}
\le C\,\|f\|_{\Lup^1(\textup{M}_\alpha w)}.
\end{equation}

In contrast, for singular integral operators the corresponding endpoint estimate fails. For example, if $\textup{H}$ denotes the Hilbert transform, Muckenhoupt and Wheeden conjectured that
\begin{equation}\label{WeakTM}
\|\textup{H}f\|_{\Lup^{1,\infty}(w)}
\le C\,\|f\|_{\Lup^1(\textup{M}w)}.
\end{equation}
This inequality was later shown to be \emph{false}; see \cite{Reg,RegTh,HM}. P\'erez~\cite{Per94Lon} (see also \cite{HP}) proved a substitute result: if T is a singular integral operator and $\varepsilon>0$, then
\begin{equation}\label{WeakbumpTM}
\|\textup{T}f\|_{\Lup^{1,\infty}(w)}
\leq {C_{\varepsilon}}\,\|f\|_{\Lup^1(\textup{M}_{\Lup(\log \Lup)^\varepsilon} w)}
\end{equation}
where $\textup{M}_{\Lup(\log \Lup)^{\varepsilon}}$ is an Orlicz maximal function (see \eqref{OrliczMax}). That is, inequality~\eqref{WeakTM} becomes valid once the maximal function applied to the weight on the right-hand side is suitably bumped.

A similar phenomenon occurs in the fractional setting. The fractional maximal function inequality~\eqref{Malweak} holds, but the fractional integral operators fail to satisfy general weak endpoint estimates. Carro, P\'erez, Soria, and Soria~\cite{CPSS} showed that the weak-type inequality
\begin{equation}\label{falseendptIal}
\|\textup{I}_\alpha f\|_{\Lup^{1,\infty}(w)}
\le C\,\|f\|_{\Lup^1(\textup{M}_\alpha w)}
\end{equation}
does not hold for arbitrary weights. They did, however, establish the following bumped replacement:
\begin{equation}\label{weakbumpIal}
\|\textup{I}_\alpha f\|_{\Lup^{1,\infty}(w)}
\le {C_{\varepsilon}}\,\|f\|_{\Lup^1\bigl(\textup{M}_\alpha(\textup{M}_{\Lup(\log \Lup)^{\varepsilon}} w)\bigr)},
\end{equation}
which holds for every weight and every $\varepsilon>0$ (again M$_{\Lup(\log \Lup)^\ep}$ is an Orlicz maximal funtion).

An unsatisfactory feature of~\eqref{weakbumpIal} is that when $\alpha=0$, the formulation fails to recover the expected endpoint behavior. Although the Riesz potential of order $\alpha=0$ is undefined, the maximal operator appearing on the right-hand side of~\eqref{weakbumpIal} should, at least formally, reduce to the maximal function in~\eqref{WeakbumpTM}. This, however, is not the case. Indeed,
\[
\textup{M}_0\bigl(\textup{M}_{\Lup(\log \Lup)^{\varepsilon}} w\bigr)
\approx \textup{M}_{\Lup(\log \Lup)^{1+\varepsilon}} w,
\]
whereas the maximal function in~\eqref{WeakbumpTM} involves only a single logarithmic bump with exponent $\varepsilon$.  {This discrepancy accounts for the loss of sharpness in the $\Lup^p$ bump conditions in \eqref{nonoptlogbump}, leading to their suboptimality.}


\subsection{The Meyers--Ziemer Theorem}

There is another closely related endpoint result that plays a fundamental role in the theory of Sobolev spaces. The Meyers–Ziemer theorem states that for a locally finite positive Borel measure 
$\mu$, the Sobolev inequality,
$$\int_{\R^n}|u|\dmu\leq C\int_{\R^n}|\nabla u|\dx, \qquad u\in \Con_c^\infty(\R^n)$$
holds if and only if the measure satisfies an upper \emph{Ahlfors-David Regular} of order $n-1$: {for any} ball $B_r(x)\subseteq \R^n$
\begin{equation}\label{Linfty}\mu(B_r(x))\leq Cr^{n-1}\qquad x\in \R^n,r>0.\end{equation}
This theorem is central to Sobolev space theory, with important applications including the characterization of Sobolev traces and the identification of the dual of $\BV(\R^n)$.

\section{Main Results}

The purpose of this article is to place these results within a unified framework and to observe that the Meyers--Ziemer theorem admits a natural and simple generalization with far-reaching consequences. In particular, we present a generalized version of the Meyers–Ziemer theorem that simultaneously recovers several known results and yields new ones.

\begin{theorem}\label{Gradient}
Let $\mu$ be a locally finite Borel measure on $\R^n$. Then there exists a dimensional constant $C=C_n$, independent of $\mu$, such that
\begin{equation}\label{newMZthm}
\int_{\R^n} |u| \, \mathrm{d}\mu
\leq
C \int_{\R^n} |\nabla u(x)| \, \textup{M}_1 \mu(x)\, \mathrm{d}x,
\qquad
u \in \Lip_c(\R^n).
\end{equation}
\end{theorem}
Several remarks are in order. First, in comparison with the classical
Meyers--Ziemer theorem, condition \eqref{Linfty} is equivalent to the
boundedness $\textup{M}_1\mu \in \textup L^\infty(\R^n)$. Moreover, whenever
$\textup{M}_1\mu < \infty$ almost everywhere with respect to Lebesgue
measure, the weight $\textup{M}_1\mu$ belongs to the Muckenhoupt class
$\textup{A}_1$ (see Theorem \ref{MalA1}). {This naturally raises the question of under which conditions on $\mu$ the estimate $\textup M_\alpha \mu < \infty$ holds.}


In the case $\alpha=0$, the condition $\textup{M}\mu<\infty$ Lebesgue almost everywhere was originally characterized by Fiorenza and Krbec \cite{FK} (see also a version for measures in \cite{BGKM}). In this work, we obtain a characterization of the finiteness of $\textup{M}_\alpha \mu$ in terms of Hausdorff measure, which provides a more natural framework. The following theorem is new and of independent interest.

\begin{theorem} \label{Malfinite}
Let $\mu$ be a locally finite Borel measure on $\R^n$ and let
$0 \le \alpha < n$. Then the following are equivalent:
\begin{enumerate}
\item There exists $x_0 \in \R^n$ such that $(\textup{M}_\alpha \mu)(x_0) < \infty$;
\item $\textup{M}_\alpha \mu < \infty$, $\mathcal H^{n-\alpha}$--almost
everywhere, that is, 
\[
\mathcal H^{n-\alpha}\bigl(\{x \in \R^n : (\textup{M}_\alpha \mu)(x) = \infty\}\bigr) = 0.
\]
\end{enumerate}
\end{theorem}
Theorem \ref{Malfinite} asserts that the fractional maximal function $\textup{M}_\alpha \mu$ is either identically infinite or finite outside a set of $\mathcal H^{n-\alpha}$–measure zero. In what follows, many of our results depend on the weight $\textup{M}_\alpha \mu$ associated with a locally finite Borel measure $\mu$. We therefore assume that $\mu$ is such that $\textup M_\al\mu$ is well defined; equivalently, that $(\textup{M}_\alpha \mu)(x_0)<\infty$ at some point $x_0\in\R^n$. Under this assumption, the weight $\textup{M}_\al\mu$ is well behaved: it is lower semicontinuous and belongs to the class $\A_1$ (see Theorem \ref{MalA1} in Section~\ref{prelim} for details).
\subsection{Consequences: BV, Capacity, and Isoperimetric Inequalities}

Inequality \eqref{newMZthm} allows for more flexibility than the original Meyers-Ziemer theorem.  For example, a direct application of the Meyer-Ziemer theorem one can take $\dmu=\frac1{|x|}\dx$ to see that $\mu$ satisfies \eqref{Linfty} and obtain the Hardy-Leray inequality.  However, inequality \eqref{newMZthm} allows for $\dmu=\frac{1}{|x|^\lambda}\dx$ for $1\leq \lambda<n$ and becomes
$$\int_{\R^n}\frac{|u(x)|}{|x|^\lambda}\dx\leq C\int_{\R^n}\frac{|\nabla u(x)|}{|x|^{1-\lambda}}\dx.$$  

The inequality \eqref{newMZthm} can be derived from a local Poincar\'e inequality established by the {fourth author} together with Franchi and Wheeden \cite{FPW}, and more recently revisited in \cite{MPW}. We emphasize, however, that our proof of Theorem~\eqref{Gradient} is considerably simpler. It proceeds via a direct modification of the original argument of Meyers and Ziemer, avoiding the intermediate local Poincar\'e machinery and yielding a more transparent approach.

There are several consequences of Theorem~\ref{Gradient}.  
We begin by noting that it admits a natural extension to weighted $\BV$ functions.  Let $w$ be a positive lower semicontinuous weight. We say that a function 
$u \in \Lup^{1}(w)$ belongs to $\BV(w)$ if
\[
|Du|_{w}(\mathbb{R}^{n})
:= \sup \left\{
\int_{\mathbb{R}^{n}} u\, \Div \boldsymbol{\Phi}\dx
:\;
\boldsymbol{\Phi} \in \Con_c^{1}(\mathbb{R}^{n},\mathbb{R}^{n}),\ 
|\boldsymbol{\Phi}(x)| \le w(x), \ \ \forall x\in \R^n
\right\}
< \infty .
\]

If $w \in \A_{1}$ (see Section~\ref{prelim} for our slightly nonstandard definition of $\A_{1}$), then functions in $\BV(w)$ may be approximated by smooth functions; see Theorem~1.2 in \cite{BGKM}. As mentioned above, if $\textup (\textup{M}_{1}\mu)(x_{0})<\infty$ for some $x_{0}\in\mathbb{R}^{n}$, then the weight $w=\textup{M}_{1}\mu$ is a lower semicontinuous A$_1$ weight (see Theorem \ref{MalA1}). In particular, we obtain the following consequence.

\begin{corollary}\label{BVextens}
Suppose $\mu\ll \mathcal L$ such that
$(\textup{M}_{1}\mu)(x_{0})<\infty$ for some $x_{0}\in\mathbb{R}^{n}$. 
Then, for every $u \in \BV(\textup{M}_{1}\mu)$,
\begin{equation*}\label{BVextensineq}
\|u\|_{\Lup^{1}(\mu)} \le C\, |Du|_{\textup{M}_{1}\mu}(\mathbb{R}^{n}).
\end{equation*}
\end{corollary}
\begin{remark} We note that extending the theory to weighted BV functions requires the underlying measure to be absolutely continuous with respect to Lebesgue measure; in other words, the result applies to weights. This technical assumption is necessitated by the use of smooth approximation, which holds $\mathcal L$-a.e.

\end{remark}

With this machinery in hand, we derive a new weighted isoperimetric inequality.  Given a lower semicontinuous weight $w$, we say that a measurable set 
$\Omega \subseteq \mathbb{R}^{n}$ has finite $w$-perimeter if 
$\mathbf{1}_{\Omega} \in \BV(w)$ (see Section~\ref{prelim} for the relevant definitions). 
The $w$-perimeter of $\Omega$ is defined by
\[
\Per(\Omega,w) := |D\mathbf 1_\Omega|_{w}.
\]
If $\textup M_1\mu$ is not identically $\infty$, then $\textup M_1\mu$ is finite outside a set of $\mathcal H^{n-1}$-measure zero. In particular, in this case the perimeter can be represented as
\[
\Per(\Omega,\textup M_1\mu)
=\int_{\partial^*\Omega} \textup M_1\mu \,\mathrm d\mathcal H^{n-1},
\]
where $\partial^*\Omega$ denotes the reduced boundary of $\Omega$ (see Section~\ref{prelim}).

\begin{corollary}\label{isoperimend}
Suppose $\Omega \subseteq \mathbb{R}^{n}$ is an open set with finite $\textup{M}_{1}\mu$-perimeter. 
Then
\begin{equation*}\label{isoperimendineq}
\mu(\Omega) \le C \Per(\Omega,\textup{M}_{1}\mu).
\end{equation*}
\end{corollary}
\begin{remark}
The constant in inequality \eqref{isoperimendineq} depends on the dimension and, more precisely, on the $\A_1$ constant of $\textup{M}_1\mu$.\end{remark}
\noindent
Corollary~\ref{isoperimend} is, in fact, equivalent to \eqref{newMZthm} via the coarea formula.

\medskip

One may also define the $\BV(w)$--capacity of a set $E \subseteq \mathbb{R}^{n}$ by
\[
\textup{Cap}(E,w)
:= \inf \left\{
\|\nabla u\|_{\Lup^{1}(w)}
:\;
u \in \Con_c^{\infty}(\mathbb{R}^{n})^{+},\ 
u \ge 1 \text{ in a neighborhood of } E
\right\}.
\]
\noindent The following result is an immediate consequence of \eqref{newMZthm} and the definition of capacity.

\begin{corollary}\label{capacity}
Let $\mu$ be a locally finite Borel measure and 
let $\Omega \subseteq \mathbb{R}^{n}$. 
Then
\begin{equation*}\label{capacityineq}
\mu(\Omega) \le C\, \textup{Cap}(\Omega,\textup{M}_{1}\mu).
\end{equation*}
\end{corollary}

\subsection{Endpoint Estimates with Riesz transforms and Maximal Functions}

Returning to the Riesz potential $\textup{I}_1$, its boundedness implies the Sobolev inequality; however, the converse is not true. Instead, one obtains a related estimate involving the Riesz transforms. The vector-valued Riesz transform $\mathbf{R}f$ is defined by
\[
\mathbf{R}f(x)
= -\nabla \textup{I}_1 f(x)
= \,\textup{p.v.}\!\int_{\R^n} \frac{c_ny}{|y|^{n+1}}\, f(x-y)\,\mathrm{d}y,
\]
where $\mathrm{p.v.}$ denotes the principal value.

Again, at the endpoint, the Riesz potential is not bounded, and the weak-type estimate \eqref{falseendptIal} fails. Instead, one obtains an alternative endpoint inequality involving the Riesz transforms on the right-hand side, which follows from a limiting argument applied to \eqref{newMZthm} with $u=\textup{I}_1 f$.

\begin{theorem}\label{RieszRieszThm}
Suppose $\mu$ is a locally finite Borel measure. Then
\begin{equation}\label{Rieszbound}
\|\textup{I}_1 f\|_{\Lup^1(\mu)}
\le C\,\|\mathbf{R}f\|_{\Lup^1(\textup{M}_1\mu)},
\qquad f \in \Con_c^\infty(\R^n).
\end{equation}
\end{theorem}
Inequality \eqref{Rieszbound} is formally equivalent to the Sobolev inequality \eqref{newMZthm}. Indeed, taking $f = (-\Delta)^{\frac12} u$ (see equation \eqref{fracderiv} for the definition of the fractional derivative) and using the identities
\[
\textup{I}_1 (-\Delta)^{\frac12} u = u
\ \ \text{and} \ \
\mathbf{R} (-\Delta)^{\frac12} u
= -\nabla \textup{I}_1 (-\Delta)^{\frac12} u
= \nabla u,
\]
yields the desired inequality.  

Notice that if we take $u=\textup{M}_1f$ in \eqref{newMZthm}, and use a result of Kinnunen and Saksman \cite{KinSak} concerning the regularity of the fractional maximal function 
$$|\nabla \textup{M}_1f|\leq c\,\textup{M}f$$
then we see that 
\begin{equation}\label{strongMM} 
\|\textup{M}_1f\|_{\Lup^1(\mu)}\leq C\|\nabla(\textup{M}_1f)\|_{\Lup^1(\textup{M}_1\mu)}\leq C\|\textup{M}f\|_{\Lup^1(\textup{M}_1\mu)}.\end{equation}
Comparing \eqref{strongMM} with \eqref{Rieszbound} we see a completely analogous inequality for maximal functions and integral operators.  In general \eqref{strongMM} extends to the full range $0\leq \al<n$ as the following theorem shows.

\begin{theorem} \label{MalMendpt}
Suppose $0 \le \alpha < n$ and $\mu$ is a locally finite Borel measure. Then
\begin{equation}\label{stMalendptgeneralmeas}
\|\textup{M}_\alpha f\|_{\Lup^{1}(\mu)} \le C\,\|\textup{M}f\|_{\Lup^1(\textup{M}_\alpha \mu)}.
\end{equation}
\end{theorem}
\begin{remark}\label{Malrem}

When $\alpha=0$ and $\mathrm{d}\mu = w\,\mathrm{d}x$, inequality~\eqref{stMalendptgeneralmeas} is a trivial consequence of the Lebesgue differentiation theorem and in particular the fact that
$$w\leq \textup Mw, \qquad \mathcal L \text{-a.e.}$$ 
However, for a general measure $\mu$, even the case $\alpha=0$ in~\eqref{stMalendptgeneralmeas} appears to be new.

We also note that Theorem~\ref{MalMendpt} is \emph{false} if the maximal operator $\textup{M}$ is removed from the right-hand side. Indeed, let $f\in \Con_c^\infty(\R^n)$ be any nontrivial function with $\supp f\subseteq B_R(0)$. Then
\[
\textup{M}_\alpha f(x)\geq \frac{C}{|x|^{n-\alpha}}, \qquad |x|\geq R.
\]
Moreover, if we take $\mathrm{d}\mu = |x|^{-\alpha}\,\mathrm{d}x$, then $\textup{M}_\alpha \mu(x)\leq C$ for all $x$, and thus
$\|f\|_{\Lup^1(\textup{M}_\alpha\mu)}<\infty$. On the other hand,
\[
\|\textup{M}_\alpha f\|_{\Lup^1(\mu)}
\geq
C\int_{|x|\geq R}\frac{1}{|x|^n}\,\mathrm{d}x
=
\infty.
\]

If $\textup{M}_\al\mu\approx 1$, then inequality \eqref{stMalendptgeneralmeas} is trivial for any non-zero function since $\|\textup{M}f\|_{\Lup^1(\R^n)}=\infty$.  However, there
exist weights for which this estimate is both meaningful and nontrivial. Given a weight $w$, Stein
\cite{St} introduced the space $\textup{M}\Lup^1(w)$, consisting of functions $f$ such that 
\[
\int_{\R^n} \textup{M}f(x)\, w(x)\dx < \infty.
\]
He showed that if $w(\R^n)=\infty$ and
\[
\int_{\R^n} \frac{w(x)}{(1+|x|)^n}\dx < \infty,
\]
then ML$^1(w)\neq\{0\}$, and $\Con_c^\infty(\R^n)$ is dense in
ML$^1(w)$. Consequently, inequality \eqref{stMalendptgeneralmeas} yields the
mapping property
\[
\textup{M}_\alpha : \textup{M}\Lup^1(\textup{M}_\alpha \mu) \longrightarrow \Lup^1(\mu).
\]
\end{remark}

Returning to the Riesz potential, inequality \eqref{Rieszbound} is closely related to Hardy space estimates for the
Riesz potential. Given an $\textup{A}_1$ weight $w$, the weighted Hardy space $\Har^1(w)$
is defined as the collection of functions $f \in \Lup^1(w)$ such that
$\mathbf R f \in \Lup^1(w)$, equipped with the norm
\[
\|f\|_{\Har^1(w)} := \|f\|_{\Lup^1(w)} + \|\mathbf R f\|_{\Lup^1(w)}.
\]
In particular, inequality \eqref{Rieszbound} implies that
\[
\textup{I}_1 : \Har^1(\textup{M}_1\mu) \longrightarrow \Lup^1(\mu).
\]

We establish the following result for general $0<\alpha<n$, which does not
appear to be present in the existing literature.

\begin{theorem}\label{Hardy}
Suppose $0<\alpha<n$ and $\mu$ is a Borel measure such that $\textup{M}_\alpha \mu < \infty$
almost everywhere with respect to Lebesgue measure. Then
\[
\|\textup{I}_\alpha f\|_{\Lup^1(\mu)} \le C\,\|f\|_{\Har^1(\textup{M}_\alpha \mu)}.
\]
\end{theorem}

When $\alpha=1$, the endpoint estimate \eqref{Rieszbound} is strictly stronger than the Hardy space bound discussed above. Interestingly, for $\alpha \neq 1$, the sharper inequality
\begin{equation}\label{conject}
\|\textup{I}_\alpha f\|_{\Lup^1(\mu)} \le C\,\|\mathbf R f\|_{\Lup^1(\textup{M}_\alpha \mu)}
\end{equation}
may fail to hold. Indeed, Spector \cite{Sp} showed that for $0<\alpha<1$, inequality \eqref{conject}—which is formally equivalent to a fractional Sobolev inequality—is false.

Table \ref{table} shows a summary of the known and unknown endpoint bounds for $\textup{M}_\al$ and $\textup{I}_\al$.

\begin{table}[h!]
\centering
\resizebox{\textwidth}{!}{%
\begin{tabular}{l c c}
\hline
\textbf{Endpoint Estimate} & \textbf{Inequality} & \textbf{Status} \\
\hline
Weak $\textup{M}_\alpha$ bound
  & $\|\textup{M}_\alpha f\|_{\Lup^{1,\infty}(\mu)} \le C\|f\|_{\Lup^1(\textup{M}_\alpha \mu)}$ 
  & {True for $0\le \alpha < n$ (\cite{St,Saw81})} \\[1.5mm]

Strong $\textup{M}_\alpha$ bound
  & $\|\textup{M}_\alpha f\|_{\Lup^{1}(\mu)} \le C\|f\|_{\Lup^1(\textup{M}_\alpha \mu)}$ 
  & {False (Remark \ref{Malrem})} \\[1.5mm]

Strong $\textup{M}_\alpha$-M bound
  & $\|\textup{M}_\alpha f\|_{\Lup^{1}(\mu)} \le C\|\textup{M}f\|_{\Lup^1(\textup{M}_\alpha \mu)}$ 
  & {True for $0\le \alpha < n$ (Theorem \ref{MalMendpt})} \\[1.5mm]

Weak $\textup{I}_\alpha$ bound
  & $\|\textup{I}_\alpha f\|_{\Lup^{1,\infty}(\mu)} \le C\|f\|_{\Lup^1(\textup{M}_\alpha \mu)}$ 
  & {False (\cite{CPSS})} \\[1.5mm]

Strong $\textup{I}_\alpha$-$\mathbf R$ bound
  & $\|\textup{I}_\alpha f\|_{\Lup^{1}(\mu)} \le C\|\mathbf R f\|_{\Lup^1(\textup{M}_\alpha \mu)}$ 
  & {True for $\alpha = 1$ (Theorem \ref{RieszRieszThm}), False for $0<\alpha< 1$ (\cite{Sp}}) \\[1.5mm]

Hardy $\textup{I}_\alpha$ bound
  & $\|\textup{I}_\alpha f\|_{\Lup^{1}(\mu)} \le C\|f\|_{\textup{H}^1(\textup{M}_\alpha \mu)}$ 
  & {True for $0<\alpha < n$ (Theorem \ref{Hardy})} \\
\hline
\end{tabular}%
}
\caption{Summary of endpoint estimates and their validity.}
\label{table}
\end{table}

\subsection{Further Consequences and Results}

The strong endpoint inequality \eqref{newMZthm} is also
closely connected to the classical Gagliardo–Nirenberg–Sobolev inequality.  First, we note that there is a natural extension of inequality \eqref{GNS} to Lorentz spaces due to 
Alvino \cite{Al}.  He showed that the GNS inequality admits a refinement in the Lorentz scale
(see Section~\ref{prelim} for definitions), namely,
\begin{equation}\label{alvinoineq}
\|u\|_{\Lup^{\frac{n}{n-1},1}(\R^n)} \le C\,\|\nabla u\|_{\Lup^1(\R^n)}.
\end{equation}
More generally, our endpoint estimates yield the following general Lorentz scale results.

\begin{corollary}\label{LorenttGradientq}
Let $\mu$ be a locally finite Borel measure. For $1 \le q \le \frac{n}{n-1}$ and
$\alpha = n - q(n-1)$, we have
\begin{equation}\label{Lorentzq}
\|u\|_{\Lup^{q,1}(\mu)} \le C \int_{\R^n} |\nabla u|\, (\textup{M}_\alpha \mu)^{\frac1q}\dx.
\end{equation}
\end{corollary}
The Lebesgue-space analogue ($\|u\|_{\Lup^q(\mu)}$ on the left-hand side) of \eqref{Lorentzq} follows from the local Poincar\'e inequality established in \cite{MPW}.  We emphasize, however, that the Lorentz-space refinement in \eqref{Lorentzq} appears to be new.  When $q=1$, inequality \eqref{Lorentzq} reduces to \eqref{newMZthm}.  
We now sketch the proof of {Corollary}~\ref{LorenttGradientq}  for $q>1$, which follows from \eqref{newMZthm} together with a simple duality argument.

Assume that $1 < q \le \frac{n}{n-1}$, and let $g \in \Lup^{q',\infty}(\mu)$ be a nonnegative function with
\[
\|g\|_{\Lup^{q',\infty}(\mu)} = 1.
\]
Then, by inequality \eqref{newMZthm}, if we let $\mathrm d\nu=g\dmu$, then we have
\[
\int_{\R^n} |u|\, g \dmu
\le C \int_{\R^n} |\nabla u|\, \textup{M}_1 \nu\dx
\]
Given any ball $B_r(x)$ we estimate
\[
r \fint_{B_r(x)} g \dmu
\le \, \frac{1}{v_nr^{n-1}}
\|g\|_{\Lup^{q',\infty}(\mu)} \,
\|\mathbf 1_{B_r(x)}\|_{\Lup^{q',1}(\mu)}
= \, \frac{\mu(B_r(x))^{\frac{1}{q}}}{v_nr^{n-1}}
= \frac{1}{v_n}\left(\frac{\mu(B_r(x))}{r^{q(n-1)}}\right)^{\frac{1}{q}}.
\]
Consequently, since $n-\al=q(n-1)$, we have
\[
\textup{M}_1 \nu
\le c\, (\textup{M}_\al \mu)^{\frac{1}{q}}.
\]

The endpoint case $q=\frac{n}{n-1}$ of inequality \eqref{LorenttGradientq} is of particular interest, as it corresponds to the scaling of the classical Gagliardo--Nirenberg--Sobolev inequality \eqref{GNS}. In fact, with $\mu=\mathcal L$, one recovers the classical GNS inequality, or more precisely its Lorentz-space refinement due to Alvino \cite{Al}. Moreover, when $q=\frac{n}{n-1}$, inequality \eqref{Lorentzq} reduces to
\begin{equation}\label{GNSmeasure}
\|u\|_{\Lup^{n',1}(\mu)}
\le C \|\nabla u\|_{\Lup^1\bigl((\textup{M}\mu)^{1-\frac1n}\bigr)},
\end{equation}
which is a Lorentz-space improvement of \eqref{GN-mu}. 

Inequality \eqref{GNSmeasure} yields the weighted Sobolev inequalities and the corresponding bounds for Riesz potentials originally obtained by Muckenhoupt and Wheeden \cite{MW}. Indeed, if $w\in \textup{A}_1$ and $\dmu = w\dx$, then \eqref{GNSmeasure} implies
\[
\|u\|_{\Lup^{\frac{n}{n-1}}(w)}
\le C \|\nabla u\|_{\Lup^1\bigl(w^{1-\frac1n}\bigr)}.
\]
Since this estimate holds uniformly for all $w\in \textup{A}_1$, the off-diagonal extrapolation theorem of Harboure, Mac\'ias, and Segovia \cite{HMS} yields
\[
\|u\|_{\Lup^{p^*}(w)}
\le C \|\nabla u\|_{\Lup^p\bigl(w^{1-\frac{p}{n}}\bigr)},
\qquad
w\in \A_{\frac{p^*}{n'}}, \quad 1<p<n.
\]
(A sharp version of this off-diagonal extrapolation with optimal bounds can be found in \cite{LMPT}. See also \cite[Corollary 1.4]{PR2} for a local version with precise bounds.) Finally, applying this inequality with $u = \textup{I}_1 f$, we obtain
\[
\|\textup{I}_1 f\|_{\Lup^{p^*}(w)}
\le C \|\nabla(\textup{I}_1 f)\|_{\Lup^p\bigl(w^{1-\frac{p}{n}}\bigr)}
= C \|\mathbf R f\|_{\Lup^p\bigl(w^{1-\frac{p}{n}}\bigr)}
\le C \|f\|_{\Lup^p\bigl(w^{1-\frac{p}{n}}\bigr)},
\]
where the last inequality follows from the boundedness of the Riesz transforms on $\Lup^p\bigl(w^{1-\frac{p}{n}}\bigr)$. This is justified by the fact that $w\in \A_{\frac{p^*}{n'}}$ guarantees the boundedness of the Riesz transforms. In fact, a more general version of this principle, which does not rely on extrapolation, is given in Theorem~\ref{wp1imply}.

When $1<q\leq \frac{n}{n-1}$, there is an alternative approach to inequality
\eqref{Lorentzq} based on the subrepresentation formula \eqref{subrepresentation}.
For $0<\beta<n$, the Riesz potential $\textup{I}_\beta$ can be written as a
convolution with a constant multiple of the kernel
\[
k_\beta(x)=|x|^{\beta-n}
\in \Lup^{\frac{n}{n-\beta},\infty}(\R^n).
\]
If $1<q\leq \frac{n}{n-\beta}$ then
\begin{equation*}
\bigl\|k_\beta\bigr\|_{\Lup^{q,\infty}(\mu)}= \sup_{\lambda>0}\lambda\,\mu\bigl(\{x\in\R^n:\ |x|^{\beta-n}>\lambda\}\bigr)^{\frac1q}=\Bigl(\sup_{r>0}\frac{\mu(B_r(0))}{r^{n-\al}}\Bigr)^{\frac{1}{q}}=c(\textup{M}_\al\mu)(0)^\frac1q
\end{equation*}
where $\al=n-q(n-\beta)$. Thus $k_\beta\in \Lup^{q,\infty}(\mu)$ if and only if $\textup M_{\al}\mu(0)<\infty$.  Applying Minkowski’s integral inequality and the calculation translated, yields the weak-type estimate
\begin{equation}\label{Minkow}
\|\textup{I}_\beta f\|_{\Lup^{q,\infty}(\mu)}
\le C \int_{\R^n} |f(y)|
\,\bigl\||\cdot-y|^{\beta-n}\bigr\|_{\Lup^{q,\infty}(\mu)}\dy
= C \int_{\R^n} |f(y)|\, (\textup{M}_\al\mu(y))^{\frac{1}{q}}\dy,
\end{equation}
Consequently,
\[
\textup{I}_\beta:\Lup^1\bigl((\textup{M}_\al\mu)^{\frac{1}{q}}\bigr)
\longrightarrow \Lup^{q,\infty}(\mu)
\]
where $1<q\leq \frac{n}{n-\beta}$, $\al=n-q(n-\beta)$.

Applying this bound to the subrepresentation formula \eqref{subrepresentation}, for $\beta=1$, $1<q\leq \frac{n}{n-1}$, and $\al=n-q(n-1)$
yields the weak Sobolev inequality
\begin{equation}\label{weakGNS}
\|u\|_{\Lup^{q,\infty}(\mu)}
\le C \|\textup{I}_1(|\nabla u|)\|_{\Lup^{q,\infty}(\mu)}
\le C \|\nabla u\|_{\Lup^1\bigl((\textup{M}_\al\mu)^{\frac1q}\bigr)}.
\end{equation}
One may then apply the weak–implies–strong principle (see Theorem~\ref{weakstrong}) to obtain the Lorentz inequality \eqref{Lorentzq}. In fact, in this measure-theoretic setting and for $q>1$, the Lorentz inequality \eqref{Lorentzq} is \emph{equivalent} to the weak-type inequality \eqref{weakGNS}, as well as to the subrepresentation formula \eqref{subrepresentation} and a corresponding isoperimetric inequality. Although each of these inequalities holds—and some are already known—the equivalence among them is new.

\begin{theorem}\label{characterizethm}
Suppose $1<q\leq \frac{n}{n-1}$ and $\al=n-q(n-1)$.  The following statements are equivalent:
\begin{enumerate}
\item The Lorentz Sobolev inequality
\[
\|u\|_{\Lup^{q,1}(\mu)}
\le C_1 \|\nabla u\|_{\Lup^1\bigl((\textup{M}_\al\mu)^{\frac1q}\bigr)}
\]
holds for all locally finite Borel measures $\mu$ and all
$u\in \Lip_c(\R^n)$;


\item The subrepresentation formula
\[
|u(x)| \le C_2\, \textup{I}_1(|\nabla u|)(x)
\]
holds for all $u\in \Lip_c(\R^n)$ and all $x\in \R^n$;

\item The weak-type Sobolev inequality
\[
\|u\|_{\Lup^{q,\infty}(\mu)}
\le C_3 \|\nabla u\|_{\Lup^1\bigl((\textup{M}_\al\mu)^{\frac1q}\bigr)}
\]
holds for all locally finite Borel measures $\mu$ and all
$u\in \Lip_c(\R^n)$;

\item The isoperimetric inequality
\begin{equation}\label{qisoper}
\mu(\Omega)^{\frac{1}{q}}
\le C_4\, \Per\!\big(\Omega,(\textup{M}_\al\mu)^{\frac1q}\big)
\end{equation}
holds for all locally finite Borel measures $\mu$ and all bounded open sets
$\Omega$ with finite $(\textup{M}_\al\mu)^{\frac1q}$-perimeter.
\end{enumerate}
Moreover, the constants are related to each other by a dimensional multiple.
\end{theorem}

\begin{remark}
It is imperative in Theorem~\ref{characterizethm} that $q>1$, particularly for the implication $(2)\Rightarrow(3)$.  
The underlying reason is that the weak Lorentz space $\Lup^{1,\infty}$ is not normed, and therefore the computations leading to \eqref{Minkow} are not valid.  
Indeed, the endpoint inequality \eqref{falseendptIal} that would result in the case $q=1$ is false.

It is natural to ask how Theorem~\ref{characterizethm} relates to the weak boundedness of the Riesz potential $\textup{I}_1$, namely,
\begin{equation}\label{weakbddI1}
\textup{I}_1 : \Lup^{1}\big((\textup M_\al\mu)^{\frac1q}\big) \to \Lup^{q,\infty}(\mu),
\end{equation}
for $1<q\leq \frac{n}{n-1}$ and $\al = n - q(n-1)$.  
The mapping property \eqref{weakbddI1} implies condition~(3) in Theorem~\ref{characterizethm}.  
Whether the Sobolev inequality itself implies the weak boundedness \eqref{weakbddI1} remains an open question.
Theorem~\ref{characterizethm}(4) generalizes the isoperimetric inequality established in \cite{PR2}, which was proved in the endpoint case \( q = \tfrac{n}{n-1} \) for sufficiently smooth sets. 

In the case \( q = \frac{n}{n-1} \), inequality \eqref{qisoper} reduces to the isoperimetric inequality \eqref{isoperintro} for sets with sufficiently smooth boundary. Moreover, when \( \mathrm d\mu = w\,\mathrm dx \) with a weight \( w \in \A_1 \), we have \( \mathrm M\mu \leq C w \). Consequently, \eqref{qisoper} recovers the weighted isoperimetric inequality first established in \cite{DS} for sets with smooth boundary and later extended to more general settings in \cite[Section 2.4]{Tur}:
\[
w(\Omega)^{\frac{n-1}{n}}
\leq
C \int_{\partial \Omega} w^{1-\frac{1}{n}} \,\mathrm d\mathcal H^{n-1}.
\]

In contrast, our result yields an entire family of isoperimetric inequalities valid for more general sets. Specifically, for \( 1 \leq q \leq \frac{n}{n-1} \), we obtain
\[
\mu(\Omega)^{\frac{1}{q}}
\leq
C \int_{\partial^* \Omega} (\mathrm M_\alpha \mu)^{\frac{1}{q}} \, \mathrm d\mathcal H^{n-1}.
\]
\end{remark}

Finally, we come full circle and return to the two-weight Sobolev inequalities \eqref{twowSob}. 
We begin by addressing the longstanding question of optimal bump conditions in the diagonal 
case $p=q$. 

To place the result in its proper context, we refer the reader to Section~\ref{prelim} for the relevant definitions of Young functions and Orlicz averages. Given a Young function~$\Psi$, one may define the associated Orlicz average over a cube~$Q$. The bump condition seeks to quantify precisely how much the standard averages in~\eqref{Apq} must be enlarged to obtain a sufficient condition.

For $1<p\le q<\infty$, Cruz-Uribe and Moen~\cite{CM} introduced the off-diagonal integrability condition $\Psi\in \textup B_{p,q}$, defined by
\begin{equation} \label{Bpq}
\int_1^\infty \frac{\Psi(t)^{\frac{q}{p}}}{t^{q+1}}\mathrm dt<\infty.
\end{equation}

In the strict off-diagonal case $1<p<q<\infty$, they showed that if $(w,v)$ is a pair of weights satisfying
\begin{equation}\label{offdiagbump}
\sup_Q \ell(Q)\,\Leb(Q)^{\frac1q-\frac1p}
\|w\|_{\Ldash^\Psi(Q)}
\left(\fint_Q v^{1-p'}\dx\right)^{\frac1{p'}}
<\infty,
\end{equation}
and if the associate Young function $\bar{\Psi}$ belongs to $\textup B_{q',p'}$, then
\[
\textup{I}_1:\Lup^p(v)\to \Lup^{q,\infty}(w).
\]
A typical Young function $\Psi$ for which $\bar{\Psi}\in  \textup B_{q',p'}$ is given by $\Psi(t)=t^q\log(\mathrm e+t)^{\frac{q}{p'}+\ep}$ for $\ep>0$.  In particular, by the truncation method Theorem \ref{weakstrong}, condition~\eqref{offdiagbump} is sufficient for the two-weight Sobolev inequality~\eqref{twowSob} in the strict off-diagonal case $p<q$.

In the diagonal case $p=q$, however, the situation is more subtle. Previously, this result was known to hold only under non-optimal logarithmic bump conditions~\eqref{nonoptlogbump}. When $p=q$, the $\textup B_{p,q}$ condition reduces to the $\textup B_p$ condition introduced by the second author~\cite{Per94}. We now obtain the following sharp result.

\begin{theorem} \label{optimalppbump}
Suppose $1<p<\infty$ and $\Psi$ is a Young function such that $\bar{\Psi}\in \textup B_{p'}$. 
If $(w,v)$ is a pair of weights satisfying
\[
\sup_Q \ell(Q)\,\|w^{\frac1p}\|_{\Ldash^\Psi(Q)}
\left(\fint_Q v^{1-p'}\dx\right)^{\frac1{p'}}<\infty,
\]
then the Sobolev inequality
\[
\|u\|_{\Lup^p(w)}\le C\|\nabla u\|_{\Lup^p(v)}
\]
holds for all $u\in \Lip_c(\R^n)$ and Riesz potential/Riesz transform bound
$$\|\textup{I}_1f\|_{\Lup^p(w)}\le C\|\mathbf R f\|_{\Lup^p(v)}$$
holds for all $f\in \Con_c^\infty(\R^n)$. In particular, one may take $\Psi(t)=t^p\log(\mathrm e+t)^{p-1+\ep}$
for any $\ep>0$.
\end{theorem}

Next, we establish the sharpness of these estimates and consider the problem of extending 
inequality \eqref{Lorentzq} to the case $p>1$. We show that such an extension cannot 
hold with a weight $w$ and the maximal function $(\textup{M}_\alpha w)^{\frac{p}{q}}$ appearing on the 
right-hand side. As a replacement, we prove that the inequality does hold for $p>1$ 
provided an appropriately bumped maximal function appears on the right-hand side. 

To this end, we introduce the Orlicz fractional maximal function associated with a Young function $\Theta$ and $0\le \alpha<n$:
\begin{equation}\label{OrliczMax}
\textup{M}_{\alpha,\Theta} f(x)
= \sup_{Q\ni x} \ell(Q)^{\alpha} \|f\|_{\Ldash^\Theta(Q)}.
\end{equation}
When $\alpha=0$, we simply write $\textup M_{0,\Theta} = \textup M_\Theta$. 

We show that, in general, a logarithmic bump is unavoidable when $p>1$, and we identify the sharp logarithmic 
maximal function for which the Sobolev inequality holds. Our main result is the following.

\begin{theorem}\label{bumpreplacep>1}
Suppose $1<p<n$, $p\le q\le p^*$, 
$\alpha=n-\frac{q}{p}(n-p)$, and $\varepsilon>0$. 
Then, for any weight $w$, the Sobolev inequality
\[
\|u\|_{\Lup^q(w)}\le C\|\nabla u\|_{\Lup^p\big((\textup{M}_{\alpha,\Theta}w)^{\frac{p}{q}}\big)}
\]
holds for all $u\in \Lip_c(\R^n)$, where $\Theta(t)=t[\log(\mathrm e+t)]^{\frac{q}{p'}+\varepsilon}.$
Moreover, the inequality fails when $\varepsilon=0$.
\end{theorem}

Returning to Alvino's Lorentz-space Sobolev inequality \eqref{alvinoineq}, or equivalently in terms of the Riesz potential and Riesz transforms, we may write
\[
\|\textup{I}_1 f\|_{\Lup^{\frac{n}{n-1},1}(\R^n)} \le C\,\|\mathbf R f\|_{\Lup^1(\R^n)}.
\]
Spector \cite{Spector} (see also \cite{SSVS}) extended this result to the full range $0<\alpha<n$, establishing that
\begin{equation}\label{spectorresult}
\|\textup{I}_\alpha f\|_{\Lup^{\frac{n}{n-\alpha},1}(\R^n)} \le C\,\|\mathbf R f\|_{\Lup^1(\R^n)}.
\end{equation}

In the context of measures, however, the behavior for $\alpha \neq 1$ remains largely unknown. If inequality \eqref{conject} were valid for $\alpha \neq 1$, it would provide a measure-theoretic analogue of Spector’s theorem. Although the general case is still open, we are able to obtain a new strong $(\tfrac{n}{n-\alpha},1)$ endpoint bound for $\alpha>1$, exploiting the strong mapping properties of the fractional integral operator $\textup{I}_\alpha$.

\begin{theorem}\label{al>1endpt}
Suppose that $1<\alpha<n$, $\varepsilon>0$, and $w$ is a weight. Then
\begin{equation*}\label{al>1bump}
\|\textup{I}_\alpha f\|_{\Lup^{\frac{n}{n-\alpha}}(w)}
\leq
C\,\|\mathbf R f\|_{\Lup^1\big((\textup{M}_\Psi w)^{1-\frac{\alpha}{n}}\big)},
\end{equation*}
where $\Psi(t)=t[\log(\mathrm e+t)]^{\frac{1}{n-\alpha}+\varepsilon}$. In particular, if $w\in \A_1$, then
\begin{equation}\label{al>1A1}
\|\textup{I}_\alpha f\|_{\Lup^{\frac{n}{n-\alpha}}(w)}
\leq
C\,\|\mathbf R f\|_{\Lup^1\big(w^{1-\frac{\alpha}{n}}\big)}.
\end{equation}
\end{theorem}

\begin{remark}
The general form of Theorem~\ref{al>1endpt} remains open. More precisely, it is unknown whether the estimate
\begin{equation}\label{lorentzconj}
\|\textup{I}_\alpha f\|_{\Lup^{\frac{n}{n-\alpha},1}(\mu)}
\leq
C\,\|\mathbf R f\|_{\Lup^1\big((\textup M\mu)^{1-\frac{\alpha}{n}}\big)},
\qquad f\in \Con_c^\infty(\R^n),
\end{equation}
holds when $\al\neq1$, even if the right-hand side is replaced by the corresponding Lebesgue space norm.

Inequality \eqref{lorentzconj} is particularly interesting for $0<\alpha<1$, since in this case it is formally equivalent to the fractional Sobolev inequality
\begin{equation}\label{fracSobconj}
\|u\|_{\Lup^{\frac{n}{n-\alpha},1}(\mu)} \le C\,\|\nabla^\alpha u\|_{\Lup^1\big((\textup M\mu)^{1-\frac{\alpha}{n}}\big)},
\end{equation}
where $\nabla^\alpha := \nabla \textup{I}_{1-\alpha}$ is the Riesz fractional gradient. Such an inequality would follow from \eqref{conject} if it were valid. Moreover, \eqref{lorentzconj} would imply \eqref{al>1A1} for $\A_1$ weights, and the full range $0<\alpha<n$. This may be viewed as an $\A_1$-weighted analogue of Spector’s inequality \eqref{spectorresult}.
\end{remark}

We end {this section} with a few observations about the fractional case with index $0<s<1$. To place our results in context, we first discuss fractional derivatives, beginning with the classical fractional Laplacian. This operator may be viewed as the inverse of $\textup I_s$ and is defined via the Fourier transform by 
\begin{equation*}\label{fracderiv} 
[(-\Delta)^{\frac{s}{2}}u]\hat{\ }(\xi)=|2\pi \xi|^s \hat{u}(\xi),\qquad u\in \Con_c^\infty(\R^n). 
\end{equation*} 

It is important to note that the fractional Sobolev inequality 
$$ \|u\|_{\Lup^1(\mu)}\leq C\|(-\Delta)^{\frac{s}2}u\|_{\Lup^1(\textup M_s\mu)} $$ 
does not hold, as it would imply the failed inequality \eqref{falseendptIal}. When $0<s<1$ (in fact, for $0<s<2$, though this is not relevant here), the fractional Laplacian $(-\Delta)^{\frac{s}{2}}u$ is a constant multiple of the singular integral 
$$\textup{p.v.}\int_{\R^n}\frac{u(y)-u(x)}{|x-y|^{n+s}}\,\mathrm{d}y. $$ 
Similarly, the integral representation for the Riesz fractional gradient $\nabla^su= \nabla \textup I_{1-s}u$ is given, up to a constant multiple, by
\begin{equation*}
\int_{\R^n}\frac{u(x)-u(y)}{|x-y|^{n+s}}\frac{x-y}{|x-y|}\,\mathrm{d}y.
\end{equation*} 
The relationship between $\nabla^s$ and $(-\Delta)^{\frac{s}{2}}$ is analogous to that between the classical gradient $\nabla$ and $(-\Delta)^{\frac12}$; this connection can be made precise via Riesz transforms and Riesz potentials:
$$\mathbf Rf=-\nabla^s \textup I_{s}f=-\textup{I}_{s}\nabla^s f \quad \text{and} \quad \mathbf R(-\Delta)^{\frac{s}{2}}u=\nabla^s u.$$ 
For more on these fractional derivatives, we refer the reader to the book by Ponce \cite{Ponce} and the papers by Comi and Stefani \cite{CSI,CSII}.

The fractional Sobolev inequality
\[
\|u\|_{\Lup^1(\mu)} \le C \|\nabla^s u\|_{\Lup^1(\textup M_s \mu)},
\]
for $0<s<1$, is formally equivalent to \eqref{conject} and is therefore false for general measures. Instead, we establish a substitute inequality involving a larger fractional derivative.

When $0 < s < 1$, the integral representations of $\nabla^s$ and $(-\Delta)^{\frac{s}{2}}$ converge absolutely for smooth functions. This allows us to define the nonlinear, positive fractional derivative
\[
\D^s u(x) = \int_{\R^n} \frac{|u(x) - u(y)|}{|x-y|^{n+s}}\,\mathrm{d}y,
\]
which pointwise dominates both of the standard fractional derivatives. More precisely,
\begin{equation*}
|(-\Delta)^{\frac{s}{2}}u(x)|+|\nabla^{s}u(x)| \le c\,\D^s u(x).
\end{equation*}

The positive derivative operator $\D^s$, introduced by Spector in \cite{Spector}, has in fact appeared implicitly in the theory of fractional Sobolev spaces for some time. In particular, the Gagliardo seminorm can be represented as
\[
[u]_{\textup W^{s,1}(\R^n)}=\|\D^s u\|_{\Lup^1(\R^n)}.
\]

Another reason why $\mathcal D^s$ is of interest is that it yields a refinement of \eqref{subrepresentation}, as shown in \cite[Lemma 2.6]{HMP1}:
\begin{equation}\label{fracgrad}
|u(x)|
\leq
(1-s)c_{n,s}\,\textup I_s(\mathcal D^s u)(x)
\leq
c_n \,\textup I_1(|\nabla u|)(x),
\qquad 0<s<1.
\end{equation}
We can further improve \eqref{fracgrad} using the Riesz fractional gradient $\nabla^s$. Indeed, using the classical Riesz transform identity $\mathbf R \cdot \mathbf R = -\mathbf I$, we see
\[
u
=
- \mathbf R \cdot \mathbf R u
=
\mathbf R \cdot \nabla \textup I_1 u
=
\textup I_s \mathbf R \cdot \nabla \textup I_{1-s} u
=
\textup I_s \mathbf R \cdot \nabla^s u,
\]
where we have used the convolution identity $\textup I_1 = \textup I_{1-s} \circ \textup I_s$ for $0<s<1$, together with the fact that the operators involved—being either derivatives or convolution operators—commute whenever $u \in \Con_c^\infty(\mathbb R^n)$. Computing the kernel of $\mathbf R\textup I_s=-\nabla \textup I_{1+s}$ yields the representation formula for $u\in \Con_c^\infty(\R^n)$:
\begin{equation}\label{fracrep}
u(x)=({\mathbf R}\textup I_s \cdot \nabla^s u)(x)
=c_{n,s}
\int_{\R^n}\frac{x-y}{|x-y|^{n-s+1}}\cdot(\nabla^s u)(y)\,\mathrm{d}y
\end{equation}
(see also \cite[Proposition 15.8, p.~246]{Ponce} for a rigorous proof of equality \eqref{fracrep}). The integral converges absolutely, and in particular we obtain the following subrepresentation formula:
\begin{equation}\label{subrieszfrac}
|u(x)|\leq \gamma_{n,s}\textup I_s(|\nabla^s u|)(x).
\end{equation}

By computing the constants explicitly, it can be seen that
\[
\gamma_{n,s}\textup I_s(|\nabla^s u|)(x)
\leq (1-s)\delta_{n,s}\textup I_s(\D^s u)(x),
\]
where $\delta_{n,s}\leq c_n$ as $s\ra 1^-$. The factor $(1-s)$ is significant because it appears naturally in the constants defining $\nabla^s$ and is consistent with the Bourgain--Brezis--Mironescu phenomenon \cite{BBM1} and \cite{BBM2}, leading to the connection with $\D^s$. Thus inequality \eqref{subrieszfrac} improves upon \eqref{subrepresentation}, and combining the above estimates with inequality \eqref{fracgrad} yields
\[
|u(x)|
\leq \gamma_{n,s}\textup I_s(|\nabla^s u|)(x)
\leq (1-s)\delta_{n,s} \textup I_s(\D^s u)(x)
\leq c_n\textup I_1(|\nabla u|)(x),
\qquad u\in \Con_c^\infty(\R^n).
\]

In \cite{MPW}, a local fractional Poincaré inequality was established for locally finite Borel measures:
\begin{equation}\label{fracPoin}\int_Q|u-u_Q|\dmu\leq C(1-s)\int_Q\int_Q\frac{|u(x)-u(y)|}{|x-y|^{n+s}}\textup M_s \mu(x)\dy \dx\end{equation}
for $0\leq s<1$ and $u\in \Lip_c(\R^n)$. By taking sufficiently large cubes in \eqref{fracPoin}, we record the following global inequality.

\begin{theorem} Suppose $\mu$ is a locally finite Borel measure and $0<s<1$. Then
$$\int_{\R^n}|u|\dmu\leq C(1-s)\int_{\R^n} \D^su\, \textup{M}_s\mu\dx, \qquad u\in \Lip_c(\R^n).$$
\end{theorem}

Finally, an argument similar to the one provided for the classical gradient in the proof of Corollary \ref{LorenttGradientq} yields the following result.

\begin{corollary} Suppose $\mu$ is a locally finite Borel measure and $0<s<1$, $1\leq q\leq \frac{n}{n-s}$, and $\al=n-q(n-s)$. Then
$$\|u\|_{\Lup^{q,1}(\mu)}\leq C(1-s)\int_{\R^n} \D^su\, (\textup{M}_\al\mu)^{\frac1q}\dx, \qquad u\in \Lip_c(\R^n).$$
\end{corollary}

Again, we remark that these observations follow from the work done in \cite{MPW} and we do not include the proofs.

\subsection{Organization of the Paper}

The paper is organized as follows. In Section~\ref{prelim}, we present the necessary background material, including basic facts about weights, Lorentz spaces, $\text{BV}$ spaces, and Orlicz spaces, as well as a key tool for our analysis: the coarea formula in weighted $\text{BV}$ spaces. Section~\ref{Master} provides a short, self-contained proof of Theorem~\ref{Gradient}, following the approach of Meyers and Ziemer, along with several related consequences. In Section~\ref{twoweight}, we revisit the two-weight Sobolev inequality \eqref{twowSob} and prove Theorems~\ref{optimalppbump}, \ref{bumpreplacep>1}, and \ref{al>1endpt}. Finally, Section~\ref{oneweight} returns to the one-weight Sobolev inequalities \eqref{weightedSob} and establishes several auxiliary results.

\section{Preliminaries}\label{prelim}

In this section, we collect background material needed throughout the paper. We work with positive Borel measures $\mu$ on $\R^n$ satisfying $\mu(K)<\infty$ for every compact set $K$. Since $\R^n$ is $\sigma$-compact, every locally finite Borel measure is a Radon measure. We denote by $\mathcal L$ the Lebesgue measure on $\R^n$.  

\subsection{Hausdorff Measure and Maximal Function Estimates}

We recall the definition of Hausdorff measure. For $s\ge 0$ and $0<\delta\le\infty$,
the \emph{Hausdorff content} of a set $E\subset\R^n$ is defined by
\[
\Haus^s_\delta(E)
:= \inf\Biggl\{\sum_k \ell(Q_k)^s
: E\subseteq\bigcup_k Q_k,\ \ell(Q_k)\le\delta\Biggr\},
\]
where the infimum is taken over all countable coverings of $E$ by cubes $Q_k$ of
side length at most $\delta$. We note that this differs slightly from the standard definition of $\Haus^s$, for example as found in \cite{EG}, in which one allows coverings by arbitrary sets. However, the two definitions are equivalent up to a constant depending only on $n$ and $s$, and we adopt the formulation above. The \emph{Hausdorff measure} of a Borel set $E$ is
then given by

\[
\Haus^s(E):=\sup_{\delta>0}\Haus^s_\delta(E).
\]
With this definition, $\Haus^s$ is an outer measure and a measure when restricted to the Borel $\sigma$-algebra. However, in this paper we will only need $\Haus^s_\infty$, which is an outer measure on $\R^n$. We note that $\Haus^s(E)=0$ if and only if $\Haus^s_\infty(E)=0$ (see \cite[p. 64]{EG}).

\subsection{Dyadic grids}

In what follows, we will need some standard dyadic techniques.  A dyadic grid $\Dy$ in $\R^n$ is a collection of cubes with side lengths $\ell(Q)=2^k$, $k\in\Z$, such that for each $k$ the subcollection
\[
\Dy^k := \{Q\in\Dy : \ell(Q)=2^k\}
\]
forms a partition of $\R^n$, and any two cubes $Q,P\in\Dy$ satisfy $Q\cap P\in\{\varnothing,Q,P\}$.

Of particular importance are the shifted dyadic grids
\begin{equation*}\label{shiftdyadic}
\Dy_t = \{\,2^k([0,1)^n + m + (-1)^k t) : k\in\Z,\ m\in\Z^n\,\},
\qquad t\in\{0,\tfrac13\}^n,
\end{equation*}
which enjoy the well-known $\tfrac13$-trick.

\begin{lemma}\label{13trick}
For any ball $B\subset \R^n$ with radius $r(B)$, there exist $t\in\{0,\tfrac13\}^n$ and a cube $Q_t\in\Dy_t$ such that $B\subset Q_t$ and $\ell(Q_t)\le 12\,r(B)$.
\end{lemma}

Given a dyadic grid $\Dy$, we define the associated dyadic fractional maximal operator for $0\le\alpha<n$ by
\[
\textup{M}_\alpha^\Dy \mu(x)
:= \sup_{Q\in\Dy} \ell(Q)^\alpha \fint_Q \dmu \,\mathbf 1_Q(x).
\]
As a consequence of Lemma~\ref{13trick}, the fractional maximal operator can be controlled pointwise everywhere by finitely many dyadic analogues.

\begin{lemma}\label{dyadicbd}
Let $\mu$ be a locally finite Borel measure and $0\le\alpha<n$. Then
\[
(\textup{M}_\alpha \mu)(x)
\le c_{n,\alpha}\sum_{t\in\{0,\frac13\}^n} \textup{M}_\alpha^{\Dy_t}\mu(x),
\qquad x\in\R^n.
\]
\end{lemma}

\subsection{$\A_1$ Weights Generated by Maximal Functions}

A \emph{weight} $w$ is a non-negative locally integrable function on $\R^n$. An $\A_1$ weight is a non-zero weight satisfying
\begin{equation}\label{A1}
(\textup{M}w)(x)\le c\,w(x), \qquad \forall x\in \R^n.
\end{equation}
This definition is slightly stronger than the standard definition of $\A_1$, which typically requires inequality~\eqref{A1} to hold only $\mathcal L$-almost everywhere. For the theory of weighted $\BV$ spaces, however, the weights must be defined everywhere, be lower semicontinuous, and satisfy condition~\eqref{A1} pointwise. In the present context, these requirements are automatically fulfilled. Indeed, if $\mu$ is a measure such that $(\mathrm{M}_1\mu)(x_0)<\infty$ for some $x_0\in\R^n$, then $w=\mathrm{M}_1\mu$ defines a weight with the desired properties, and most of the results relevant to this paper apply in this setting. In fact, we have the following theorem, originally due to Sawyer \cite{Saw81two} in the weighted case and later established in full generality by Hurri-Syrj\"anen, Mart\'inez-Perales, Pérez, and V\"ah\"akangas \cite[Lemma~3.6]{HMPV}.

\begin{theorem}\label{MalA1}
Suppose $0\le \alpha<n$ and $\mu$ is a locally finite Borel measure. If there exists $x_0\in \R^n$ such that $(\textup{M}_\alpha\mu)(x_0)<\infty$, then $\textup{M}_\alpha\mu$ is defined for all $x\in \R^n$ (possibly taking the value $+\infty$ on a set of $\Haus^{n-\al}$ measure zero), is lower semicontinuous, and satisfies
\[
(\textup{M}_\alpha\mu)^r \in \A_1,
\quad \text{for all } 0\le r<\frac{n}{n-\alpha}.
\]
\end{theorem}

\subsection{Lorentz and Orlicz Spaces}

Given a measure $\mu$ on $\R^n$ and $1\le r,s<\infty$, the Lorentz space $\Lup^{r,s}(\mu)$ consists of all measurable functions $f$ such that
\[
\|f\|_{\Lup^{r,s}(\mu)}^s
= r\int_0^\infty t^{s-1}\,
\mu\big(\{x\in \R^n:|f(x)|>t\}\big)^{\frac{s}{r}}\,\mathrm dt<\infty.
\]
When $s=\infty$, the weak-$\Lup^r$ space $\Lup^{r,\infty}(\mu)$ is defined by
\[
\|f\|_{\Lup^{r,\infty}(\mu)}
=\sup_{t>0} t\,\mu\big(\{x\in \R^n:|f(x)|>t\}\big)^{\frac{1}{r}}.
\]
These spaces satisfy the well-known embeddings
\[
\Lup^{r,1}(\mu)\subseteq \Lup^{r,r}(\mu)=\Lup^{r}(\mu)\subseteq \Lup^{r,\infty}(\mu).
\]

For the background on Orlicz spaces, we refer the reader to the book \cite[Chapter 5]{CMP}. A Young function $\Psi$, is a continuous, increasing function satisfying $\Psi(0)=0$ and
\[
\frac{\Psi(t)}{t}\to\infty \qquad \text{as } t\to\infty.
\]
The associate Young function $\bar{\Psi}$ is defined by
\[
\bar{\Psi}(t)=\sup_{s>0}(st-\Psi(s)).
\]
Given a Young function $\Psi$ and a cube $Q$, the normalized Orlicz average is defined by
\[
\|f\|_{\Ldash^\Psi(Q)}
=\inf\left\{\lambda>0:\fint_Q\Psi\Big(\frac{|f(x)|}{\lambda}\Big)\,\dx\le 1\right\}.
\]
When $\Psi(t)=t^a\log(\mathrm e+t)^b$ we write 
$$\|f\|_{\Ldash^\Psi(Q)}=\|f\|_{\Ldash^a(\log \Ldash)^b(Q)}.$$
\subsection{Weighted $\BV$ Spaces and Perimeter}

We now collect several facts concerning weighted $\BV$ spaces; for further details, we refer to \cite{BGKM}. Let $w\in$ A$_1$ be a positive lower semicontinuous function defined on $\R^n$. 
For $u\in \Lup^1_{\text{loc}}(\R^n)$, the weighted variation is defined by
\[
\textup{Var}_w u
=\sup\left\{\int_{\R^n}u\,\Div \boldsymbol{\Phi}\,\dx
:\boldsymbol{\Phi}\in \Lip_c(\R^n,\R^n),\ |\boldsymbol{\Phi}|\le w\right\}.
\]
We say that $u\in \BV(w)$ if $u\in \Lup^1(w)$ and $\textup{Var}_w u<\infty$, and we equip this space with the norm
\[
\|u\|_{\BV(w)}=\|u\|_{\Lup^1(w)}+\textup{Var}_w u.
\]
If $u\in \BV(w)$, then $u\in \BV_{\text{loc}}(\R^n)$ and there exists a finite Borel measure $|Du|_w$ such that
\begin{equation}\label{absolcont}
\mathrm d|Du|_w=w\,\mathrm d|Du|
\end{equation}
where $|Du|$ is the total variation measure arising from $u\in \BV_{\text{loc}}(\R^n)$.

Moreover,
\[
\textup{Var}_w u=|Du|_w(\R^n)
=\int_{\R^n} w\,\mathrm d|Du|.
\]

Given a set $\Omega \subseteq \R^n$, we say that $\Omega$ has finite weighted perimeter if $\mathbf{1}_\Omega \in \BV(w)$.  
In this case, the weighted perimeter (or $w$-perimeter) is defined by
\[
\Per(\Omega,w)
:= \textup{Var}_w(\mathbf{1}_\Omega)
= \sup\left\{
\int_{\Omega} \Div \boldsymbol{\Phi}\dx
:\ \boldsymbol{\Phi} \in \Lip_c(\R^n,\R^n),\ |\boldsymbol{\Phi}(x)| \le w(x), \forall x\in \R^n
\right\}.
\]

When $w \equiv \mathbf 1$, we write $\Per(\Omega)=\Per(\Omega,\mathbf 1)$ for the classical perimeter (see \cite{EG}) and for the associated weighted perimeter, we denote  
\[
|\partial \Omega|_w := |D\mathbf 1_\Omega|_w.
\] 
In general, if $\Omega$ has finite $w$-perimeter, then it has locally finite classical perimeter in $\R^n$; conversely, since a positive lower semicontinuous function is bounded below on compact sets, a bounded set with finite classical perimeter also has finite $w$-perimeter (see \cite{BGKM}). Moreover, since 
\begin{equation}\label{restrtict}
|D\mathbf 1_\Omega| = \mathcal H^{n-1} \resmes \partial^* \Omega,
\end{equation}
where $\partial^*\Omega$ denotes the reduced boundary of $\Omega$, that is, the set of points $x \in \R^n$ at which a well-defined measure-theoretic unit normal exists (which coincides with the topological boundary if $\partial \Omega$ is $\Con^1$), it follows from \eqref{absolcont} and \eqref{restrtict} that
\[
\Per(\Omega,w) = \int_{\partial^*\Omega} w \,\mathrm d\mathcal H^{n-1}.
\]

We will require the following weighted coarea formula; see \cite{Cam} for a proof of this version with general weights.
\begin{lemma}\label{coarea}
Suppose $w$ is a weight and $u\in \BV(w)$ is non-negative. For $t>0$, set $E_t=\{u>t\}$. Then $E_t$ has finite $w$ perimeter for $\Leb$-a.e. $t\in \R$ and
\[
|Du|_w(\R^n)
=\int_0^\infty \Per(E_t,w)\,\mathrm dt
=\int_0^\infty |\partial E_t|_w(\R^n)\,\mathrm dt
=\int_0^\infty \int_{\R^n} w\,\mathrm d|D\mathbf 1_{E_t}|\,\mathrm dt.
\]
In particular, if $u\in \Lip_c(\R^n)$ then
$$\|\nabla u\|_{\Lup^1(w)}=\int_0^\infty\Per(E_t,w)\,\mathrm dt=\int_0^\infty \int_{\R^n} w\,\mathrm d|D\mathbf 1_{E_t}|\,\mathrm dt.$$
\end{lemma}

Finally, we will need the following unweighted relative isoperimetric inequality. The proof follows from the local Poincar\'e inequality; we refer the reader to Evans and Gariepy~\cite[p.~190]{EG}.
\begin{lemma}\label{lowerbdiso}
Suppose $E\subseteq \R^n$ is a bounded set with finite perimeter. Then for any ball $B_r(x)\subseteq \R^n$,
\begin{equation}\label{iso}
\min\{\mathcal L(B_r(x)\cap E),\mathcal L(B_r(x)\cap E^c)\}^{\frac{n-1}{n}}
\le C\,|D\mathbf 1_E|(B_r(x)).
\end{equation}
\end{lemma}



\section{Proofs of main results}\label{Master}

\subsection{The extension of the Meyer--Ziemer theorem and its consequences}

We first establish our main result, Theorem~\ref{Gradient}, together with its consequences, Corollaries~\ref{BVextens} and~\ref{isoperimend}. We also prove Theorem~\ref{characterizethm}, which characterizes the endpoint Sobolev inequalities in the range $1<q\leq \frac{n}{n-1}$.

\begin{proof}[Proof of Theorem~\ref{Gradient}]
Let $E\subseteq \R^n$ be a bounded open set of finite perimeter. Since $E$ is open, for every $x\in E$ we have
\[
\lim_{r\to\infty}\frac{\mathcal L(B_r(x)\cap E)}{\mathcal L(B_r(x))}=0,
\qquad
\lim_{r\to 0^+}\frac{\mathcal L(B_r(x)\cap E)}{\mathcal L(B_r(x))}=1.
\]
Consequently, there exists $r_x>0$ such that
\[
\frac{\mathcal L(B_{r_x}(x)\cap E)}{\mathcal L(B_{r_x}(x))}=\frac12,
\qquad\text{that is,}\qquad
\mathcal L(B_{r_x}(x)\cap E)=\frac12\,\mathcal L(B_{r_x}(x)).
\]
It follows that
\[
\mathcal L(B_{r_x}(x)\cap E^c)
=\mathcal L(B_{r_x}(x))-\mathcal L(B_{r_x}(x)\cap E)
=\frac12\,\mathcal L(B_{r_x}(x)).
\]
Applying Lemma~\ref{lowerbdiso}, we obtain for each $x\in E$ the estimate
\[
r_x^{n-1}\le C\,|D\mathbf 1_E|(B_{r_x}(x)).
\]

Now let $u\in \Lip_c(\R^n)$. For each $t>0$, the level set $E_t=\{|u|>t\}$ is open, bounded, and has finite perimeter. Let $K\subseteq E_t$ be compact. Then
\[
K\subseteq \bigcup_{x\in K} B_{r_x}(x).
\]
By compactness and the Vitali covering lemma, there exists a finite disjoint collection of balls $B_j=B_{r_j}(x_j)$, with $x_j\in K$, such that
\[
K\subseteq \bigcup_{j=1}^N 3B_j.
\]
Therefore,
\begin{align*}
\mu(E_t)
&\le \sum_{j=1}^N \mu(3B_j)
= c_n\sum_{j=1}^N \frac{\mu(3B_j)}{r_j^{\,n-1}}\, r_j^{\,n-1} \\
&\le c_n\sum_{j=1}^N \frac{\mu(3B_j)}{r_j^{\,n-1}}\,
|D\mathbf 1_{E_t}|(B_j)
\le c_n\sum_{j=1}^N \int_{B_j} \textup{M}_1\mu\,\mathrm d|D\mathbf 1_{E_t}| \\
&\le c_n\int_{\R^n} \textup{M}_1\mu\,\mathrm d|D\mathbf 1_{E_t}|
= c_n\,|\partial E_t|_{\textup{M}_1\mu}(\R^n),
\end{align*}
where we used that for every $x\in B_j$,
\[
\textup{M}_1\mu(x)\ge c_n\,\frac{\mu(3B_j)}{r_j^{\,n-1}}.
\]

Finally, integrating over $t\in(0,\infty)$ and applying the coarea formula with weight $w=\textup{M}_1\mu$ (Lemma~\eqref{coarea}), we obtain
\[
\int_{\R^n} |u|\,\mathrm d\mu
= \int_0^\infty \mu(E_t)\,\mathrm dt
\le C\int_0^\infty \int_{\R^n} \textup{M}_1\mu\,\mathrm d|D\mathbf 1_{E_t}|\,\mathrm dt
\le C\int_{\R^n} |\nabla u|\,\textup{M}_1\mu\,\dx.
\]
This concludes the proof.
\end{proof}

\begin{proof}[Proof of Corollary~\ref{BVextens}]
Let $u \in \BV(\textup{M}_1 w)$. By Theorem~1.2 in~\cite{BGKM}, there exists a sequence $\{u_k\}\subseteq \Con_c^\infty(\R^n)$ such that $u_k\to u$ $\mathcal L$-a.e. and
\[
\limsup_{k\to\infty} \|\nabla u_k\|_{\Lup^1(\textup{M}_1 w)}
\le C\, |Du|_{\textup{M}_1 w}(\R^n),
\]
where $C$ depends only on the dimension and the $\A_1$ constant of $\textup{M}_1 w$ (see Theorem~\ref{MalA1}).

By Fatou’s lemma—this is precisely where we require $w$ to be a weight, since the convergence holds $\mathcal L$-a.e.—and inequality~\eqref{newMZthm}, we obtain
\[
\|u\|_{\Lup^1(w)}
\le \liminf_{k\to\infty} \|u_k\|_{\Lup^1(w)}
\le C \limsup_{k\to\infty} \|\nabla u_k\|_{\Lup^1(\textup{M}_1 w)}
\le C\, |Du|_{\textup{M}_1 w}(\R^n).
\]
\end{proof}

We now turn to Corollary~\ref{isoperimend}. In contrast to the extension result above, the isoperimetric inequality holds for general locally finite Borel measures $\mu$, provided that $\Omega$ is a bounded open set with finite $\textup{M}_1\mu$-perimeter.

\begin{proof}[Proof of Corollary~\ref{isoperimend}]
Let $\vp$ be a standard mollifier supported in $B_1(0)$, nonnegative, radially decreasing, and normalized so that $\int \vp=1$. Set $\vp_\varepsilon(x)=\varepsilon^{-n}\vp(x/\varepsilon)$ and
\[
u_\varepsilon=\vp_\varepsilon * \mathbf{1}_\Omega.
\]
Since $\Omega$ is open, every $x\in\Omega$ is a point of continuity of $\mathbf{1}_\Omega$, and hence
\[
u_\varepsilon(x)\to 1 \qquad \text{for all } x\in\Omega \quad \text{as } \varepsilon\to 0.
\]
By Fatou’s lemma and inequality~\eqref{newMZthm}, we obtain
\[
\mu(\Omega)
\le \liminf_{\varepsilon\to 0} \int_{\R^n} u_\varepsilon\,\dmu
\le C \liminf_{\varepsilon\to 0} \int_{\R^n} |\nabla u_\varepsilon|\,\textup{M}_1\mu\,\dx.
\]

Let $w=\textup{M}_1\mu$. Then
\[
\int_{\R^n} |\nabla u_\varepsilon|\, w\,\dx
= \textup{Var}_{w}(u_\varepsilon)
= \sup_{\substack{\boldsymbol{\Phi}\in \Lip_c(\R^n;\R^n)\\ |\boldsymbol{\Phi}|\le w}}
\int_{\R^n} u_\varepsilon \,\Div \boldsymbol{\Phi}\,\dx.
\]
For such $\boldsymbol{\Phi}$, Fubini’s theorem yields
\[
\int_{\R^n} u_\varepsilon \,\Div \boldsymbol{\Phi}\,\dx
= \int_\Omega \Div(\vp_\varepsilon * \boldsymbol{\Phi})\,\dy.
\]
Since $\vp$ is radially decreasing,
\[
|\vp_\varepsilon * \boldsymbol{\Phi}|
\le \vp_\varepsilon * |\boldsymbol{\Phi}|
\le \vp_\varepsilon * w
\le \textup{M}w
\le C w,
\]
where we used that convolution with a radially decreasing mollifier is controlled by the Hardy--Littlewood maximal operator and that $w\in\A_1$. Therefore,
\[
\int_{\R^n} u_\varepsilon \,\Div \boldsymbol{\Phi}\,\dx
\le \textup{Var}_{Cw}(\mathbf{1}_\Omega)
= C\,\Per(\Omega,\textup{M}_1\mu).
\]
Taking the supremum over admissible $\boldsymbol{\Phi}$ completes the proof.
\end{proof}

We conclude this section with the proof of Theorem~\ref{characterizethm}.

\begin{proof}

To prove that (1) implies (2), let $\mu = \delta_{x_0}$ be a point mass at $x_0 \in \R^n$. Then
\[
\|u\|_{\Lup^{q,1}(\delta_{x_0})}=|u(x_0)|,
\]
and
\begin{equation*}
\int_{\R^n}|\nabla u(x)|(\textup{M}_\al\delta_{x_0})(x)^{\frac1q}\dx
=c\int_{\R^n}\frac{|\nabla u(x)|}{|x-x_0|^{n-1}}\dx
=c\,\textup{I}_1(|\nabla u|)(x_0).
\end{equation*}
The implication $(2)\Rightarrow(3)$ follows from inequality~\eqref{weakGNS}. The proof of $(3)\Rightarrow(4)$ is essentially identical to that of Corollary~\ref{isoperimend}. Finally, for $(4)\Rightarrow(1)$, let $E_t=\{|u|>t\}$ for $t>0$, which is open, bounded, and has the necessary finite perimeter, provided $u\in \Lip_c(\R^n)$. By the coarea formula,
\[
\|u\|_{\Lup^{q,1}(\mu)}
=q\int_0^\infty \mu(E_t)^{\frac1q}\dt
\le C\int_0^\infty \Per\big(E_t,(\textup{M}_\al\mu)^{{\frac1q}}\big)\dt
=\int_{\R^n}|\nabla u|(\textup{M}_\al\mu)^{\frac1q}\dx.
\]
\end{proof}


\subsection{Endpoint Inequalities}

We now establish the endpoint inequalities involving operators, starting with Theorem~\ref{RieszRieszThm}, followed by Theorems~\ref{MalMendpt} and~\ref{Hardy}.

\begin{proof}[Proof of Theorem~\ref{RieszRieszThm}]
Let $\mu$ be a locally finite Borel measure and set $w=\textup{M}_1\mu\in\A_1$, assuming $\textup{M}_1\mu$ is not identically infinite. If $f\in \Con_c^\infty(\R^n)$, then $\textup{I}_1 f\in \Con^\infty(\R^n)$.

Let $\vp\in \Con_c^\infty(\R^n)$ satisfy $0\le \vp\le 1$, $\vp\equiv 1$ on $B_1(0)$, and $\vp\equiv 0$ outside $B_2(0)$, and set $\vp_k(x)=\vp(x/k)$. Then $u_k=\vp_k\,\textup{I}_1 f\in \Con_c^\infty(\R^n)$ and
\[
\nabla u_k
=(\textup{I}_1 f)\nabla \vp_k+\vp_k\nabla \textup{I}_1 f
=(\textup{I}_1 f)\nabla \vp_k-\vp_k\mathbf{R}f.
\]
By inequality~\eqref{newMZthm},
\[
\|u_k\|_{\Lup^1(\mu)}
\le C\|\nabla u_k\|_{\Lup^1(w)}
\le C\Big(\|(\textup{I}_1 f)\nabla \vp_k\|_{\Lup^1(w)}+\|\vp_k\mathbf{R}f\|_{\Lup^1(w)}\Big).
\]
The second term is bounded by $\|\mathbf{R}f\|_{\Lup^1(w)}$. For the first term,
\begin{align*}
\|(\textup{I}_1 f)\nabla \vp_k\|_{\Lup^1(w)}
&\le \frac{C}{k}\|\mathbf 1_{B_{2k}(0)\setminus B_k(0)}\textup{I}_1 f\|_{\Lup^1(w)} \\
&\le \frac{C}{k}w(B_{2k}(0))^{1/n}
\|\mathbf 1_{B_{2k}(0)\setminus B_k(0)}\textup{I}_1 f\|_{\Lup^{n',\infty}(w)} \\
&\le C(\textup{M}w(0))^{1/n}
\|\mathbf 1_{B_{2k}(0)\setminus B_k(0)}\textup{I}_1 f\|_{\Lup^{n',\infty}(w)}.
\end{align*}

Now, using the estimate 
\[
\|\textup{I}_1 f\|_{\Lup^{n',\infty}(\mu)} \le c_n \|f\|_{\Lup^1((\textup M\mu)^{\frac{1}{n'}} )},
\]
which is just \eqref{Minkow} for $\beta=1$ and $q=\frac{n}{n-1}$ and the fact that $f\in \Lup^1(w^{\frac{1}{n'}})$ with $w\in\A_1$ we have $\textup{I}_1 f\in \Lup^{n',\infty}(w)$;  thus the final term tends to zero as $k\to\infty$.

Finally, applying Fatou’s lemma and using $\vp_k(x)\textup{I}_1 f(x)\to \textup{I}_1 f(x)$ for all $x$, we obtain
\[
\|\textup{I}_1 f\|_{\Lup^1(\mu)}
\le C\|\mathbf{R}f\|_{\Lup^1(\textup{M}_1\mu)}.
\]
\end{proof}
\begin{proof}[Proof of Theorem~\ref{MalMendpt}]
By Lemma~\ref{dyadicbd}, it suffices to consider a general dyadic operator $\textup{M}_\al^\Dy$ and a function $f\in \Lup^\infty_c(\R^n)$. We use the well-known sparse bounds for $\textup{M}_\al$ (see \cite[Theorem 5.1]{CM}).  

Specifically, there exists a subfamily of cubes $\mathcal S \subseteq \Dy$ such that for each $Q\in \mathcal S$ there is a measurable subset $E_Q \subseteq Q$ with $|E_Q|\ge \frac12|Q|$, the sets $\{E_Q: Q\in \mathcal S\}$ are pairwise disjoint, and
\[
\textup M^\Dy_\al f(x) \le c\sum_{Q\in \mathcal S} \ell(Q)^\al \fint_Q |f| \dx \,\mathbf 1_{E_Q}(x), \quad \forall x\in \R^n,
\]
where $\mathcal S=\mathcal S(f)$ may depend on $f$, but $c$ depends only on the dimension.  

Integrating with respect to $\mu$, we then have
\begin{align*}
\int_{\R^n} \textup M_\al^\Dy f \,\dmu
&\le C\sum_{Q\in \mathcal S} \ell(Q)^\al \fint_Q |f| \dx \, \mu(E_Q) \\
&\le C\sum_{Q\in \mathcal S} \Big(\fint_Q |f| \dx\Big) \Big(\ell(Q)^\al \frac{\mu(Q)}{|Q|}\Big) |E_Q| \\
&\le C\sum_{Q\in \mathcal S} \int_{E_Q} \textup M f(x) \textup M_\al \mu(x) \dx \\
&\le C \int_{\R^n} \textup M f(x) \textup M_\al \mu(x) \dx.
\end{align*}
This completes the proof.
\end{proof}

\begin{proof}[Proof of Theorem~\ref{Hardy}]
Let $w = \textup{M}_1 \mu \in \A_1$.  For such weights, the Riesz transform and atomic characterizations of the weighted Hardy space \(\textup{H}^1(w)\) are equivalent. Consequently, it suffices to verify boundedness on atoms.  Specifically, let \(a \in L^\infty_c(\R^n)\) be supported on a ball \(B = B_R(x_0)\) and satisfy
\begin{equation}\label{atomsize}
\int_B a \,\mathrm{d}x = 0, \qquad \|a\|_{\Lup^\infty(\R^n)} \le w(B)^{-1}.
\end{equation}
We will show that
\[
\|\textup{I}_\alpha a\|_{\Lup^1(w)} \le C,
\]
where \(C\) is independent of the atom \(a\).  

We split the integral of $|\textup I_\al a|$ over $\R^n$ into a local and a global part:
\[
\int_{\R^n} |\textup I_\al a| \,\dmu = \int_{3B} |\textup I_\al a| \,\dmu + \int_{(3B)^c} |\textup I_\al a| \,\dmu.
\]

\medskip
Using the cancellation of $a$ and standard kernel estimates, we estimate the global part as follows:
\begin{align*}
\int_{(3B)^c} |\textup I_\al a| \,\dmu
&= C \int_{(3B)^c} \Big|\int_B a(y) \Big(\frac{1}{|x-y|^{n-\al}} - \frac{1}{|x-x_0|^{n-\al}}\Big) \dy \Big| \dmu(x) \\
&\le C \int_{(3B)^c} \int_B |a(y)| \frac{|y-x_0|}{|x-x_0|^{n-\al+1}} \dy \, \dmu(x) \\
&\le C \int_B |a(y)| \int_{(3B)^c} \frac{|y-x_0|}{|x-y|^{n-\al+1}} \dmu(x) \, \dy \\
&\le C \int_B |a(y)| \textup{M}_\al \mu(y) \dy.
\end{align*}
By \eqref{atomsize}, this yields
\[
\int_{(3B)^c} |\textup I_\al a| \,\dmu \le C.
\]

\medskip
For the local part, let $x\in 3B$,
\[
|\textup I_\al a(x)| \le \frac{1}{w(B)} \int_B \frac{1}{|x-y|^{n-\al}} \dy \le \frac{C}{w(B)} R^\al,
\]
so that
\[
\int_{3B} |\textup I_\al a| \,\dmu \le \frac{C R^\al \mu(3B)}{w(B)} \le C \frac{1}{w(B)} \int_B \textup M_\al \mu \, \dx \le C.
\]

Combining the local and global estimates completes the proof.
\end{proof}


\subsection{The Domain of the Fractional Maximal Operator}

Next, we prove Theorem~\ref{Malfinite}, which characterizes the domain of the fractional maximal operator acting on measures. As a preliminary step, we establish an endpoint estimate for the fractional maximal function in terms of Hausdorff measure. This result will play a key role in the proof of Theorem~\ref{Malfinite} and is of independent interest.

\begin{theorem}\label{weakHaus}
Let $\mu$ be a finite Borel measure on $\R^n$ and let $0 \le \alpha < n$. Then there exists a constant $C>0$ such that for every $\lambda>0$,
\[
\Haus_\infty^{n-\alpha}\bigl(\{x \in \R^n : (\textup{M}_\alpha \mu)(x) > \lambda\}\bigr)
\le \frac{C}{\lambda}\,\mu(\R^n).
\]
In particular,
\[
\Haus_\infty^{n-\alpha}\bigl(\{x \in \R^n : (\textup{M}_\alpha \mu)(x) = \infty\}\bigr) = 0,
\]
and hence $(\textup{M}_\alpha \mu)(x) < \infty$ for $\Haus^{n-\alpha}$-almost every $x \in \R^n$.
\end{theorem}


\begin{proof}
Let $\mu$ be a finite Borel measure. By Lemma~\ref{dyadicbd}, it suffices to prove the estimate for the dyadic fractional maximal operator $\textup{M}_\alpha^{\Dy}$ associated with an arbitrary dyadic grid $\Dy$.

Fix $\lambda>0$ and consider the level set
\[
E_\lambda := \{x\in\R^n : (\textup{M}_\alpha^{\Dy}\mu)(x) > \lambda\}.
\]
Since $\mu(\R^n)<\infty$, we may select a collection of maximal dyadic cubes $\{Q_j\}\subset\Dy$ such that
\[
\ell(Q_j)^\alpha \fint_{Q_j} \dmu > \lambda.
\]
These cubes are pairwise disjoint and satisfy
\[
E_\lambda = \bigsqcup_j Q_j.
\]

For each $Q_j$ we have $\ell(Q_j)^{n-\alpha} < \mu(Q_j)/\lambda$, and therefore
\[
\sum_j \ell(Q_j)^{n-\alpha}
< \frac{1}{\lambda}\sum_j \mu(Q_j)
= \frac{\mu(E_\lambda)}{\lambda}
\le \frac{\mu(\R^n)}{\lambda}.
\]
By the definition of Hausdorff content, this implies
\[
\Haus_\infty^{n-\alpha}(E_\lambda)
\le \frac{\mu(\R^n)}{\lambda},
\]
which completes the proof.
\end{proof}

We now prove a slightly more general characterization, following the original work of 
 Fiorenza and Krbec \cite{FK}. 

\begin{theorem}\label{Malfinitevar}
Let $0\le\alpha<n$ and let $\mu$ be a locally finite Borel measure on $\R^n$. The following statements are equivalent:
\begin{enumerate}
\item $\textup{M}_\alpha\mu$ is finite $\Haus^{n-\alpha}$-almost everywhere;
\item there exists $x_0\in\R^n$ such that $(\textup{M}_\alpha\mu)(x_0)<\infty$;
\item there exists $x_0\in\R^n$ such that
\[
\limsup_{R\to\infty} \frac{\mu(B_R(x_0))}{R^{n-\alpha}} < \infty;
\]
\item there exists a constant $K\ge0$ such that for every $x\in\R^n$,
\[
\limsup_{R\to\infty} \frac{\mu(B_R(x))}{R^{n-\alpha}} = K.
\]
\end{enumerate}
\end{theorem}

\begin{proof}[Proof of Theorem~\ref{Malfinitevar}]
The implications (1)$\Rightarrow$(2)$\Rightarrow$(3) are immediate.

Assume (3) holds for some $x_0\in\R^n$ and set
\[
K := \limsup_{R\to\infty} \frac{\mu(B_R(x_0))}{R^{n-\alpha}} < \infty.
\]
For any $x\in\R^n$, writing $d=|x-x_0|$ and noting that $B_R(x)\subset B_{R+d}(x_0)$, we obtain
\[
\limsup_{R\to\infty} \frac{\mu(B_R(x))}{R^{n-\alpha}}
\le \limsup_{R\to\infty} \frac{(R+d)^{n-\alpha}}{R^{n-\alpha}}
\frac{\mu(B_{R+d}(x_0))}{(R+d)^{n-\alpha}}
= K.
\]
Exchanging the roles of $x$ and $x_0$ yields (4).

Finally, assume (4). In particular,
\[
\limsup_{R\to\infty} \frac{\mu(B_R(0))}{R^{n-\alpha}} = K < \infty.
\]
Choose $N_0\in\mathbb{N}$ such that $\mu(B_R(0))/R^{n-\alpha}\le 2K$ for all $R\ge N_0$, and define
\[
E_\infty := \{x\in\R^n : \textup{M}_\alpha\mu(x)=\infty\}.
\]
Since $\R^n=\bigcup_{N=N_0}^\infty B_N(0)$, it suffices to show
$\mathcal H^{n-\alpha}(E_\infty\cap B_N(0))=0$ for each $N\ge N_0$.

Decompose $\mu$ as $\mu_N := \mu\resmes B_{2N}(0)$ and
$\nu_N := \mu\resmes B_{2N}(0)^c$, so that
$\textup{M}_\alpha\mu \le \textup{M}_\alpha\mu_N + \textup{M}_\alpha\nu_N$.
Since $\mu_N$ is a finite measure, Theorem~\ref{weakHaus} yields
\[
\mathcal H^{n-\alpha}(\{x\in B_N(0) : \textup{M}_\alpha\mu_N(x)=\infty\})=0.
\]
For $x\in B_N(0)$, we have $\nu_N(B_r(x))=0$ for $r<N$, while for $r\ge N$,
$B_r(x)\subset B_{2r}(0)$, and hence
\[
\frac{\nu_N(B_r(x))}{r^{n-\alpha}}
\le \frac{\mu(B_{2r}(0))}{r^{n-\alpha}}
\le 2^{n-\alpha+1}K.
\]
Thus $\textup{M}_\alpha\nu_N(x)$ is uniformly bounded on $B_N(0)$, and the conclusion follows.
\end{proof}

\section{Two weight Sobolev inequalities}\label{twoweight}

We now collect the necessary two weight machinery to prove Theorems~\ref{optimalppbump}, \ref{bumpreplacep>1}, and \ref{al>1endpt}. Before doing so, we first establish some general results on two weight Sobolev inequalities, which are of independent interest. 

We begin by discussing the weak–to–strong phenomenon. In particular, the coarea formula allows one to upgrade a weak endpoint Sobolev inequality to a strong one. Indeed, suppose that the following weighted weak Sobolev inequality holds:

\begin{equation}\label{weakgen}
\|u\|_{\Lup^{q,\infty}(w)} \le C \|\nabla u\|_{\Lup^1(v)}
\end{equation}
holds for all $u \in \Lip_c(\R^n)$, some $1 \le q < \infty$ and weights $(w,v)$.  By density of smooth functions, this inequality extends to the weighted space $\BV(v)$ and hence applies to sets of finite $v$-perimeter.  
Taking $u=\mathbf{1}_E$ in \eqref{weakgen} (which requires a limiting procedure similar to that in the proof of Corollary \ref{isoperimend}) for such a set $E$ yields the isoperimetric inequality
\[
w(E)^{\frac1q} \le C\,\Per(E,v).
\]
Now let $E_t=\{x\in\R^n : |u(x)|>t\}$.  
Then, by the coarea formula,
\begin{equation*}
\|u\|_{\Lup^{q,1}(w)}
= q\int_0^\infty w(E_t)^{\frac1q}\,\dt
\le C\int_0^\infty \Per(E_t,v)\,\dt
= C\int_{\R^n} |\nabla u|\,v\,\dx.
\end{equation*}

When $p>1$, the preceding argument no longer applies.  
In this case one may instead use the truncation method due to Maz'ya \cite{Maz} (see also \cite{Haj2}).

\begin{theorem}\label{weakstrong}
Let $1 \le p,q < \infty$, and let $(w,v)$ be a pair of weights such that the weak Sobolev inequality
\begin{equation}\label{weakSob}
\|u\|_{\Lup^{q,\infty}(w)} \le C \|\nabla u\|_{\Lup^{p}(v)}
\end{equation}
holds for all $u \in \Lip_c(\R^n)$.  
Then the weighted Lorentz inequality
\begin{equation}\label{weakimpLorentz}
\|u\|_{\Lup^{q,p}(w)} \le C \|\nabla u\|_{\Lup^{p}(v)}
\end{equation}
also holds for all $u \in \Lip_c(\R^n)$.
\end{theorem}

\begin{remark}

\begin{itemize}

\item The truncation method is well known in the theory of Sobolev spaces.  
However, to the best of our knowledge, Theorem~\ref{weakstrong} is new at this level of generality, particularly in the presence of both weights and Lorentz spaces.  (see Koskela and Onninen \cite{KO} for a one-weight version on finite domains)

\item One can verify that the proof remains valid when the density on the left-hand side is replaced by a general locally finite Borel measure $\mu$.  
However, the density on the right-hand side must be absolutely continuous with respect to $\Leb$, a requirement already implicit in the coarea formula \eqref{coarea}.

\item If $p \le q$, then inequality \eqref{weakimpLorentz} implies the corresponding Lebesgue Sobolev inequality
\[
\|u\|_{\Lup^{q}(w)} \le C \|\nabla u\|_{\Lup^{p}(v)},
\]
by the usual embedding of Lorentz spaces, namely $\Lup^{q,p}(w)\subseteq \Lup^{q}(w)$.

\item The subrepresentation formula \eqref{subrepresentation} shows that the weak boundedness
\[
\textup{I}_1 : \Lup^p(v) \to \Lup^{q,\infty}(w)
\]
implies inequality \eqref{weakSob}, which then self-improves to \eqref{weakimpLorentz}. It is natural to ask whether the converse implication holds.
\end{itemize}
\end{remark}

\begin{proof}
Since $u \in \Lip_c(\R^n)$ implies $|u| \in \Lip_c(\R^n)$ and $|\nabla |u|| \le |\nabla u|$, we may assume without loss of generality that $u \ge 0$.  
For each $k \in \Z$, define the truncation
\[
\tau_k u(x) =
\begin{cases}
0, & u(x) \le 2^k, \\
u(x) - 2^k, & 2^k < u(x) \le 2^{k+1}, \\
2^k, & u(x) > 2^{k+1}.
\end{cases}
\]
Since $\Lip_c(\R^n)$ is preserved under truncation, we have $\tau_k u \in \Lip_c(\R^n)$ and
\[
|\nabla \tau_k u| \le |\nabla u|\,\mathbf{1}_{\{2^k < u \le 2^{k+1}\}}.
\]
Moreover,
\[
\{x : u(x) > 2^{k+1}\} \subseteq \{x : \tau_k u(x) > 2^{k-1}\}.
\]

We now discretize the Lorentz norm:
\begin{multline*}
\|u\|_{\Lup^{q,p}(w)}^p
= q\int_0^\infty t^{p-1} w(\{x : u(x) > t\})^{\frac{p}{q}}\,dt \\
\le c_{p,q}\sum_k 2^{(k+1)p} w(\{x : u(x) > 2^{k+1}\})^{\frac{p}{q}}
\le c_{p,q}\sum_k 2^{kp} w(\{x : \tau_k u(x) > 2^{k-1}\})^{\frac{p}{q}}.
\end{multline*}
Applying the weak Sobolev inequality \eqref{weakSob} to $\tau_k u$ and summing over $k$, we obtain
\[
\|u\|_{\Lup^{q,p}(w)}^p
\le C\sum_k \int_{\R^n} |\nabla \tau_k u|^p v\dx 
\le C\sum_k  \int_{\{2^k<u\leq 2^{k+1}\}} |\nabla u|^p v\dx 
\le C\int_{\R^n} |\nabla u|^p v\dx,
\]
which completes the proof.
\end{proof}

As noted in the introduction, two-weight inequalities for the operators \(\textup{I}_\alpha\) and \(\textup{M}_\alpha\) have been extensively studied. Complete characterizations in terms of testing conditions were first obtained for \(\textup{M}_\alpha\) in \cite{Saw82} and subsequently for \(\textup{I}_\alpha\) in \cite{Saw84,Saw88}. For our purposes, we will rely on the following optimal sufficient bump conditions, originally due to P\'erez \cite{Per94} (see also \cite{Per95two}). We will employ the refined version in \cite{CM}, which uses the \(\textup B_{p,q}\) condition (see \eqref{Bpq}) rather than the \(\textup B_p\) condition, providing a sharper result in the off-diagonal case \(p<q\).

\begin{lemma}[Theorem 3.4 \cite{CM}] \label{strongbumpM}
Suppose $1<p\leq q<\infty$, \((w,v)\) is a pair of weights, and \(\Psi\) is a Young function with \(\bar{\Psi}\in \textup B_{p,q}\). If the weights satisfy
\[
\sup_Q \, \ell(Q)^\alpha \, \mathcal{L}(Q)^{\frac{1}{q}-\frac{1}{p}}
\Bigl(\fint_Q w \dx \Bigr)^{\frac{1}{q}}
\|v^{-\frac{1}{p}}\|_{\Ldash^\Psi(Q)} < \infty,
\]
then 
\[
\|\textup{M}_\alpha f\|_{\Lup^q(w)} \le C \|f\|_{\Lup^p(v)}.
\]
\end{lemma}

\begin{lemma}[Theorem 2.1 \cite{Per94}, Theorem 3.6 \cite{CM}] \label{strongbumpI}
Suppose $1<p\leq q< \infty$, \((w,v)\) is a pair of weights, and \((\Phi,\Psi)\) is a pair of Young functions with \(\bar{\Phi}\in \textup B_{q',p'}\) and \(\bar{\Psi}\in \textup B_p\). If the weights satisfy
\begin{equation}\label{twosidebump}
\sup_Q \, \ell(Q)^\beta \, \mathcal{L}(Q)^{\frac{1}{q}-\frac{1}{p}}
\|w^{\frac{1}{q}}\|_{\Ldash^\Phi(Q)} \, 
\|v^{-\frac{1}{p}}\|_{\Ldash^\Psi(Q)} < \infty,
\end{equation}

then
\[
\|\textup{I}_\beta f\|_{\Lup^q(w)} \le C \|f\|_{\Lup^p(v)}.
\]
\end{lemma}
\begin{remark}
For the fractional maximal function \(\textup{M}_\alpha\), it suffices to bump only one of the weights. In contrast, for \(\textup{I}_\alpha\), both weights must be bumped, as the example in \cite{CM} demonstrates. In the case \(p<q\), one can separate the bumps for \(\textup{I}_\alpha\) (see \cite{CM}). However, for \(p=q\), it is currently unknown whether the bumps may be separated. This issue is closely related to two-weight weak-type \((p,p)\) inequalities for \(\textup{I}_\alpha\).
\end{remark}

Finally, we state a general theorem, which will be used in the proofs of Theorems \ref{bumpreplacep>1} and \ref{al>1endpt}. This result follows from Lemma \ref{strongbumpI} and appears to be new.

\begin{theorem} \label{generalmax}
Suppose $0<\beta<n$, $1<p<\frac{n}{\beta}$, and $p\leq q\leq \frac{np}{n-\beta p}$. Let $w$ be a weight. Then
\begin{equation}\label{generalIalbd}
\|\textup{I}_\beta f\|_{\Lup^q(w)} \leq C \|f\|_{\Lup^p\big((\textup{M}_{\alpha,\Theta}w)^{\frac{p}{q}}\big)},
\end{equation}
where 
\[
\Theta(t) = t[\log(\mathrm e + t)]^{\frac{q}{p'} + \varepsilon}, 
\quad \alpha = n - \frac{q}{p}(n-\beta p).
\]
Moreover, this inequality is sharp in the sense that it fails if $\varepsilon = 0$.
\end{theorem}

\begin{remark} 
The restrictions on $p$ and $q$ guarantee that $0 \leq \alpha < n$.
\end{remark}

\begin{proof} 
By Theorem~\ref{strongbumpI}, it suffices to verify that the pair $(w,v)$, with 
\[
v = (\textup{M}_{\alpha,\Theta} w)^{\frac{p}{q}},
\] 
satisfies condition~\eqref{twosidebump} for suitable Young functions $(\Phi,\Psi)$ such that $\bar\Phi \in \textup B_{q',p'}$ and $\bar\Psi \in \textup B_p$.

Define
\[
\Phi(t) = t^q[\log(\mathrm e + t)]^{\frac{q}{p'} + \varepsilon}.
\]
Then
\[
\bar\Phi(t) \approx \frac{t^{q'}}{[\log(\mathrm e + t)]^{\frac{q'}{p'} + \frac{q'}{q}\varepsilon}} \in \textup B_{q',p'},
\]
and since $\Phi(t^{1/q}) \approx \Theta(t)$,
\[
\|w^{\frac{1}{q}}\|_{\Ldash^\Phi(Q)} \approx \|w\|_{\Ldash^\Theta(Q)}^{\frac{1}{q}}.
\]

Next, let $r = 1 + \varepsilon$ and set $\Psi(t) = t^{r p'}$, so that $\bar\Psi(t) \approx t^{(r p')'} \in \textup B_p$. (The choice of bump here is flexible; a power bump suffices.)

With these choices and $\alpha = n - \frac{q}{p}(n-\beta p)$, we estimate
\begin{multline*} 
\ell(Q)^\beta \mathcal{L}(Q)^{\frac{1}{q}-\frac{1}{p}} 
\|w^{\frac{1}{q}}\|_{\Ldash^\Phi(Q)} \|v^{-\frac{1}{p}}\|_{\Ldash^\Psi(Q)} 
\leq C \big(\ell(Q)^{\alpha} \|w\|_{\Ldash^\Theta(Q)}\big)^{\frac{1}{q}} \|(\textup{M}_{\alpha,\Theta} w)^{-\frac{1}{q}}\|_{\Ldash^\Psi(Q)} \\
\leq C \|(\textup{M}_{\alpha,\Theta} w)^{\frac{1}{q}} (\textup{M}_{\alpha,\Theta} w)^{-\frac{1}{q}}\|_{\Ldash^\Psi(Q)} \leq C,
\end{multline*}
since for every $x \in Q$,
\[
\ell(Q)^{\alpha} \|w\|_{\Ldash^\Theta(Q)} \le (\textup{M}_{\alpha,\Theta} w)(x).
\]
Thus, condition~\eqref{twosidebump} holds, and we obtain
\[
\|\textup{I}_\beta f\|_{\Lup^q(w)} \leq C \|f\|_{\Lup^p\big((\textup{M}_{\alpha,\Theta} w)^{\frac{p}{q}}\big)}.
\]

To see the sharpness, assume by contradiction that \eqref{generalIalbd} holds with $\varepsilon = 0$, i.e.,
\[
\Theta(t) = t[\log(\mathrm e + t)]^{\frac{q}{p'}}.
\]
Let $w = \mathbf{1}_B$, where $B = B_1(0)$ is the unit ball in $\mathbb{R}^n$. Then \eqref{generalIalbd} gives
\[
\|\textup{I}_\beta f\|_{\Lup^q(B)} \leq C \|f\|_{\Lup^p\big((\textup{M}_{\alpha,\Theta}(\mathbf{1}_B))^{\frac{p}{q}}\big)}.
\]
A duality argument shows
\[
\|\textup{I}_\beta g\|_{\Lup^{p'}((\textup{M}_{\alpha,\Theta}(\mathbf{1}_B))^{-\frac{p'}{q}})} \leq C \|g\|_{\Lup^{q'}(B)}.
\]
Taking $g = \mathbf{1}_B$ yields
\[
\|\textup{I}_\beta (\mathbf{1}_B)\|_{\Lup^{p'}((\textup{M}_{\alpha,\Theta}(\mathbf{1}_B))^{-\frac{p'}{q}})} \leq C.
\]
Since $|x-y| \leq |x| + 1$ for $y \in B$, we have
\[
\textup{I}_\beta(\mathbf{1}_B)(x) \ge \frac{v_n}{(1+|x|)^{n-\beta}}.
\]
Moreover, for $|x|>2$,
\[
\textup{M}_{\alpha,\Theta}(\mathbf{1}_B)(x) \approx \frac{(\log|x|)^{\frac{q}{p'}}}{|x|^{n-\alpha}}.
\]
By the definition of $\alpha$,
\[
(n-\alpha)\frac{p'}{q} = (n-\beta)p' - n.
\]

Hence,
\begin{multline*}
\int_{\mathbb{R}^n} \textup{I}_\beta(\mathbf{1}_B)(x)^{p'} \textup{M}_{\alpha,\Theta}(\mathbf{1}_B)(x)^{-\frac{p'}{q}} \dx
\ge C \int_{|x|>2} \frac{|x|^{(n-\beta)p' - n}}{(1+|x|)^{(n-\beta)p'} \log|x|} \dx \\
\ge C \int_{|x|>2} \frac{1}{|x|^n \log|x|} \dx = \infty,
\end{multline*}
which is a contradiction.
\end{proof}

\begin{proof}[Proof of Theorem \ref{optimalppbump}] The proof is a direct consequence of Theorems~\ref{Gradient} and~\ref{strongbumpM}. Let $p>1$ and $u\in \Lip_c(\R^n)$. Using the identification
$(\Lup^p(w))^* = \Lup^{p'}(w^{1-p'}),$ we may write
\[
\|u\|_{\Lup^p(w)}
= \sup\left\{ \int_{\R^n} |u g|\,\dx : \|g\|_{\Lup^{p'}(w^{1-p'})}=1 \right\}.
\]
Fix such a function $g$. By Theorem~\ref{Gradient} and H\"older's inequality,
\[
\int_{\R^n} |u g|\,\dx
\le C \int_{\R^n} |\nabla u|\, \textup{M}_1 g\,\dx
\le C \|\nabla u\|_{\Lup^{p}(v)} \|\textup{M}_1 g\|_{\Lup^{p'}(v^{1-p'})}.
\]
To control the maximal term, we apply Theorem~\ref{strongbumpM} to the fractional maximal operator on $\Lup^{p'}(w^{1-p'})$ with the weight pair $(v^{1-p'},w^{1-p'})$. This yields
\[
\|\textup{M}_1 g\|_{\Lup^{p'}(v^{1-p'})}
\le C \|g\|_{\Lup^{p'}(w^{1-p'})},
\]
provided that
\[
\sup_Q \ell(Q)\left(\fint_Q v^{1-p'}\,\dx\right)^{\frac1{p'}}
\|w^{-(1-p')\frac1{p'}}\|_{\Ldash^\Psi(Q)} < \infty
\]
for some $\Psi$ with $\bar{\Psi}\in \textup B_{p'}$. Since $-(1-p')\frac1{p'}=\frac1p$ this condition is precisely the hypothesis of the theorem. Combining the estimates, we obtain
\[
\int_{\R^n} |u g|\,\dx
\le C \|\nabla u\|_{\Lup^{p}(v)},
\]
which completes the proof.

\end{proof}

\begin{proof}[Proof of Theorem \ref{bumpreplacep>1}] By the subrepresentation formula~\eqref{subrepresentation} it suffices to check
\begin{equation*}\label{I1bdd}
\textup I_1:\Lup^p\big((\textup{M}_{\alpha,\Theta}w)^{\frac{p}{q}}\big)\to \Lup^q(w),
\end{equation*}
where $\Theta(t)=t[\log(\mathrm e+t)]^{\frac{q}{p'}+\varepsilon}$ for $\varepsilon>0$. (Weak-type boundedness would suffice, but this pair of weights satisfies the condition for strong boundedness.) However, this is true by Theorem~\ref{generalmax} with $\beta=1$. Thus 
\[
\|u\|_{\Lup^q(w)}
\le C\|\textup I_1(|\nabla u|)\|_{\Lup^q(w)}
\le C\|\nabla u\|_{\Lup^p\big((\textup{M}_{\alpha,\Theta}w)^{\frac{p}{q}}\big)}.
\]
Thus the desired result holds.  The sharpness, i.e. that we cannot take $\ep=0$, also follows from Theorem~\ref{generalmax}.  Indeed, suppose 
$$\|u\|_{\Lup^q(w)}
\le C\|\nabla u\|_{\Lup^p\big((\textup{M}_{\alpha,\Theta}w)^{\frac{p}{q}}\big)}$$
held with $\Theta(t)=t[\log(\mathrm e+t)]^{\frac{q}{p'}}$.  Then we could take $u=\textup I_1f$ would imply
$$\|\textup I_1f\|_{\Lup^q(w)}
\le C\|\mathbf Rf\|_{\Lup^p\big((\textup{M}_{\alpha,\Theta}w)^{\frac{p}{q}}\big)}\leq C\|f\|_{\Lup^p\big((\textup{M}_{\alpha,\Theta}w)^{\frac{p}{q}}\big)}$$
where we have used that the Riesz transforms are bounded on $\Lup^p\big((\textup{M}_{\alpha,\Theta}w)^{\frac{p}{q}}\big)$ since the weight
$$(\textup{M}_{\alpha,\Theta}w)^{\frac{p}{q}}\in\A_1.$$
However, this inequality is false by the example in the proof of Theorem \ref{generalmax}.
\end{proof}

\begin{proof}[Proof of Theorem~\ref{al>1endpt}]
Let $1<\alpha<n$. We begin with the convolution identity
\[
\textup I_\alpha=\textup I_{\alpha-1}\circ \textup I_1.
\]
We then apply the strong boundedness of $\textup I_{\alpha-1}$ from Theorem~\ref{generalmax}. Set
\[
\beta=\alpha-1, \qquad p=\frac{n}{n-1}, \qquad q=\frac{n}{n-\alpha},
\]
which satisfy the hypotheses of Theorem~\ref{generalmax}. In this case,
\[
n-\frac{q}{p}(n-\beta p)
= n-\frac{n-1}{n-\alpha}\Bigl(n-\frac{(\alpha-1)n}{n-1}\Bigr)
=0.
\]
Therefore,
\[
\|\textup I_\alpha f\|_{\Lup^{\frac{n}{n-\alpha}}(w)}
= \|\textup I_{\alpha-1}(\textup I_1 f)\|_{\Lup^{\frac{n}{n-\alpha}}(w)}
\le C \|\textup I_1 f\|_{\Lup^{\frac{n}{n-1}}\big((\textup M_{\Theta}w)^{\frac{n-\alpha}{n-1}}\big)},
\]
where
\[
\Theta(t)=t[\log(\mathrm e+t)]^{\frac{q}{p'}+\varepsilon}
= t[\log(\mathrm e+t)]^{\frac{1}{n-\alpha}+\varepsilon}.
\]
Next, by Theorem~\ref{LorenttGradientq},
\[
\|\textup I_1 f\|_{\Lup^{\frac{n}{n-1}}\big((\textup M_{\Theta}w)^{\frac{n-\alpha}{n-1}}\big)}
\le C \|\mathbf R f\|_{\Lup^1\big((\textup M(\textup M_{\Theta}w)^{\frac{n-\alpha}{n-1}})^{\frac{n-1}{n}}\big)}.
\]
Since $\frac{n-\alpha}{n-1}<1$, the weight $(\textup M_{\Theta}w)^{\frac{n-\alpha}{n-1}}$ belongs to $\A_1$ (this follows from a refinement of Theorem~\ref{MalA1} for Orlicz maximal operators). Consequently,
\[
\textup M\big((\textup M_{\Theta}w)^{\frac{n-\alpha}{n-1}}\big)(x)
\le C \big[\textup M_{\Theta}w(x)\big]^{\frac{n-\alpha}{n-1}},
\]
and combining the exponents yields
\[
\|\textup I_\alpha f\|_{\Lup^{\frac{n}{n-\alpha}}(w)}
\le C \|\mathbf R f\|_{\Lup^1\big((\textup M_{\Theta}w)^{1-\frac{\alpha}{n}}\big)}.
\]

\end{proof} 

\section{One weight Sobolev inequalities}\label{oneweight}

We return to the one weight Sobolev inequality \eqref{weightedSob},
\begin{equation}\label{onesobolev}
\left(\int_{\R^n} |u|^{p^*} w\dx\right)^{\frac1{p^*}}
\leq C\left(\int_{\R^n}|\nabla u|^p w^{1-\frac{p}{n}}\dx\right)^\frac1p,
\end{equation}
where we assume that $u\in \Lip_c(\R^n)$, $1\leq p<n$, and $p^*=\frac{np}{n-p}$. For a fixed $1\leq p<n$, let $\mathcal W_p$ denote the class of weights for which inequality \eqref{onesobolev} holds for all $u\in \Lip_c(\R^n)$.

The class $\A_{\frac{p^*}{n'}}$, consisting of weights such that
\[
\sup_Q\left(\fint_Q w\dx\right)^{\frac{1}{p^*}}
\left(\fint_Qw^{-\frac{p'}{p^*}}\dx\right)^{\frac1p}<\infty,
\]
characterizes the weighted boundedness of the Riesz potential $\textup{I}_1$, in the sense that
\[
\textup I_1:\Lup^p(w^{1-\frac{p}{n}})\rightarrow \Lup^{p^*}(w)
\iff  w \in \A_{\frac{p^*}{n'}}, \qquad 1<p<n,
\]
and
\[
\textup I_1:\Lup^1(w^{1-\frac{1}{n}})\rightarrow \Lup^{\frac{n}{n-1},\infty}(w)
\iff  w \in \A_1.
\]
By the subrepresentation formula \eqref{subrepresentation} (together with Theorem \ref{weakstrong} in the weak endpoint case), we have
\[
\A_{\frac{p^*}{n'}}\subseteq \mathcal W_p.
\]
While $\A_{\frac{p^*}{n'}}$ characterizes the boundedness of $\textup I_1$, it is not clear in general whether it also characterizes the weighted Sobolev inequality \eqref{weightedSob}.

As mentioned in the introduction, another approach uses strong $\A_\infty$ weights as defined by David and Semmes \cite{DS}. The class of strong $\A_\infty$ weights, sometimes denoted by $\A_\infty^*$, consists of all {doubling weights $w$ for which the quasidistance function
\[
\delta(x,y)=w(B_{x,y})^{\frac1n},
\]
where $B_{x,y}$ denotes the smallest Euclidean ball containing $x$ and $y$, it is equivalent to a metric on $\R^n$.} {This means, there exists a positive constant $C$ and a distance function $d$ on $\mathbb{R}^n$ such that 
\[
C^{-1}\delta(x,y) \leq d(x,y) \leq C \delta(x,y).
\]
Actually, this condition implies that strong $\A_\infty$ weights belongs to the class $\A_\infty$ (see \cite[Proposition 3.4]{Sem} and \cite{Ge}). Moreover, it can be shown that $\A_1 \subseteq \A^*_\infty$, but for each $p>1$, there is an $\A_p$ weight which is not strong $\A_\infty$; for more details about this see \cite{Sem} and \cite{DS}.}

In their work, David and Semmes showed that if $w\in \A_\infty^*$, then the weighted Sobolev inequality \eqref{weightedSob} holds. To demonstrate the inequality in the range $1<p<n$ they proved a weighted subrepresentation formula in the same spirit of  \eqref{subrepresentation} and for the case $p=1$ they used a corresponding weighted isoperimetric inequality.

Consequently,
\[
\A_{\frac{p^*}{n'}} \,\cup\, \A_\infty^* \subseteq \mathcal W_p.
\]

On the other hand, when $w$ is doubling, which we denote by $w\in \textup{DB}$, membership in $\A_\infty$ is in fact a necessary condition. To see this, test \eqref{weightedSob} with a cutoff function $\varphi$ that equals $1$ on a ball $B$, is supported in $2B$, and satisfies $|\nabla \varphi|\le C/r$, where $r$ is the radius of $B$. Then
\[
\left(\int_B w\dx\right)^{\frac1{p^*}}
\le \left(\int_{\R^n} |\varphi|\, w\dx\right)^{\frac1{p^*}}
\le C \left(\int_{\R^n} |\nabla \varphi|^p\, w^{1-\frac{p}{n}}\dx\right)^{\frac1p}
\le \frac{C}{r}\left(\int_{2B} w^{1-\frac{p}{n}}\dx\right)^{\frac1p}.
\]
Rewriting this estimate yields
\[
\fint_B w\dx \le C\left(\fint_B w^{\frac{p}{p^*}}\dx\right)^{\frac{p^*}{p}}.
\]
This shows that $w^{\frac{p}{p^*}}$ belongs to the reverse H\"older class $\textup{RH}_{\frac{p^*}{p}}$, which is known to be equivalent to $w\in \A_\infty$ (see \cite[Corollary 6.2]{SW}). In particular, we have shown that
\[
\mathcal W_p\cap\textup{DB}\subseteq \A_\infty.
\]

As a remark, we note that, as in the unweighted case, it suffices to establish the endpoint inequality corresponding to \( p = 1 \). We therefore record the following theorem, whose proof follows the same argument as in the unweighted setting (see \cite[Chapter 4]{EG}; see also \cite[Theorem 1.21]{PR1}).

\begin{theorem}\label{wp1imply}
If $w$ is a weight such that \eqref{weightedSob} holds for $p=1$ and all $u\in \Lip_c(\R^n)$, then the inequality holds for any $1<p<n$ and all $u\in \Lip_c(\R^n)$. That is, $\mathcal W_1\subseteq \mathcal W_p$ for $1<p<n$.
\end{theorem}

\begin{proof}
Let $w$ be a weight such that
\begin{equation}\label{endptoneweight}
\|u\|_{\Lup^{\frac{n}{n-1}}(w)}\leq C\|\nabla u\|_{\Lup^1(w^{1-\frac1n})},
\qquad u\in \Lip_c(\R^n).
\end{equation}
Fix $1<p<n$ and set $q=\frac{p^*}{n'}>1$. Let $u\in \Lip_c(\R^n)$ and assume without loss of generality that $u$ is not identically zero. Since $u$ is bounded and $q>1$, we have $|u|^q\in \Lip_c(\R^n)$ and
\[
|\nabla |u|^q|\leq |u|^{q-1}|\nabla u|.
\]
Applying the endpoint inequality \eqref{endptoneweight} to $|u|^q$ yields
\begin{multline*}
\left(\int_{\R^n}|u|^{p^*} w\dx\right)^{\frac{1}{n'}}
=\left(\int_{\R^n}(|u|^{q})^{n'} w\dx\right)^{\frac{1}{n'}}
\leq C\int_{\R^n}|u|^{q-1}|\nabla u| w^{1-\frac1n}\dx \\
= C\int_{\R^n}|u|^{q-1}w^{\frac1{p'}}|\nabla u|w^{\frac1p-\frac1n}\dx
\leq C\left(\int_{\R^n}|\nabla u|^pw^{1-\frac{p}{n}}\dx\right)^{\frac1p}
\left(\int_{\R^n}|u|^{(q-1)p'}w\dx\right)^{\frac{1}{p'}}.
\end{multline*}
Since
\[
(q-1)p'=\left(\frac{p^*}{n'}-1\right)p'=p^*,
\]
reorganizing the inequality yields
\[
\left(\int_{\R^n}|u|^{p^*} w\dx\right)^{\frac{1}{n'}-\frac1{p'}}
\leq C\left(\int_{\R^n}|\nabla u|^pw^{1-\frac{p}{n}}\dx\right)^{\frac1p},
\]
which is exactly the desired Sobolev inequality since 
$$\frac{1}{n'}-\frac1{p'}=\frac1p-\frac1n=\frac1{p^*}.$$
\end{proof}

Finally, all of this leads to the following conjecture.

\begin{conj}\label{Ainftyconj}
Suppose $w\in \A_\infty$. Then the Sobolev inequality \eqref{weightedSob} holds, that is, $\A_\infty\subseteq \mathcal W_1$.
\end{conj}

To prove Conjecture \eqref{Ainftyconj}, it would suffice to show that $\A_\infty$ is sufficient for the weak endpoint inequality
\[
\|u\|_{\Lup^{\frac{n}{n-1},\infty}(w)}
\leq C\|\nabla u\|_{\Lup^1(w^{1-\frac1n})}.
\]






\end{document}